\documentclass[]{article}
\usepackage[utf8]{inputenc}
\usepackage[english]{babel}
\usepackage{amsmath,amsthm,verbatim,amssymb,amsfonts,amscd, graphicx}
\usepackage{graphics}
\usepackage{authblk}
\theoremstyle{plain}
\newtheorem{theorem}{Theorem}
\newtheorem{corollary}{Corollary}
\newtheorem{lemma}{Lemma}

\theoremstyle{definition}
\newtheorem{definition}{Definition}
\theoremstyle{claim}

\newcommand*\diff{\mathop{}\!\mathrm{d}}
\newcommand*\E{\mathbb{E}}

\let\al=\alpha

\let\ve=\epsilon

\title{Concentration inequalities for sums of random variables, each having power bounded tails}
\author{Oleksii Omelchenko}
\author{Andrei A. Bulatov \thanks{This work was supported by an NSERC Discovery grant.}}
\affil{School of Computing Science, Simon Fraser University, Canada}
\affil[]{ \textit{\{oomelche, abulatov\}}@sfu.ca } 
\begin{document}

\date{}
\maketitle

\begin{abstract}
 %Heavy-tailed distributions are much less studied than their exponentially decaying counterparts, and one of possible reasons is the lack of higher order moments. However, recent studies suggest that many industrial computational problems may be modelled best with non-exponentially decaying distributions; hence, the study of heavy-tailed pdfs may be beneficial for understanding why many applied problems are easily tractable.
 
 In this work we present concentration inequalities for the sum $S_n$ of independent integer-valued not necessary indentically distributed random variables, where each variable has tail function that can be bounded by some power function with exponent $-\alpha$. We show that when $0<\alpha\leq 1$, then the sum does not have finite expectation, however, with high probability we have that $|S_n|=O\left(n^{1/\alpha}\right)$. When $\alpha>1$, then the sum $S_n$ is concentrated around its mean. 
 
 Since the r.vs. that constitute the sum has tails, which can be bounded by some power function, it follows that results of this paper are applicable to a wide range of different distributions, including the exponentially decaying ones. 
\end{abstract}

\section{Introduction}

Random combinatorial structures and related computational 
problems have been studied for decades. The majority of
research have been focused on structures constructed 
or selected according to some sort of uniform distribution, as this
is usually the most natural approach, and also most of the techniques
work best in this case. However, many applied and real world problems
are best modelled with non-uniform distributions 
\cite{Ansotegui09:structure,Ansotegui09:towards,Broder00:graph,%
	Kumar00:random,Newman05:power,Clauset09:power,Barabasi99:typical}. 
So, in recent years there have been
an increasing amount of work on structures sampled from less 
standard distributions, including heavy-tailed ones, such as the 
power law distribution \cite{Friedrich17:bounds,Cooper14:typical}, or 
somewhat arbitrary distributions 
\cite{Cooper07:random,Molloy95:critical,Molloy98:size}. 

One of the main difficulties in studying distributions similar to the 
power law is that many standard tools are not available for such 
distributions. Concentration bounds like Chernoff Bound or Azuma 
inequality may not apply because heavy-tailed distribution lacks 
higher moments, or even the second or the first moment. More 
sophisticated approaches like Fourier analysis may be lost as well, 
as the Friedgut's satisfiability threshold \cite{Friedgut99:sharp} 
demonstrates, that it is notoriously difficult to generalize beyond 
near-uniform distributions.

Heavy tail distributions have been studied in probability theory for
decades \cite{BrysonMauriceC.1974HDPa}. In particular, some (though not very strong) 
concentration bounds can be found in \cite{Borovkov08:asymptotic}. 
It is therefore 
somewhat surprizing that such bounds are (to our best knowledge) 
not used in the computer science literature, instead substituted by 
ad hoc methods or some results working in special cases 
\cite{Hopcroft14:data}. 
Apart from relative obscurity of these results from probability theory, 
a reason for that may be that the existing bounds tend to be 
proved in a very general setting, which, although being very 
powerful, often applies to continuous random variables, or does not
give the kind of bounds needed in combinatorics.

In this paper we consider concentration bounds for sums of 
random variables, possibly with heavy tails. The paper is mostly 
based on the results of \cite{Borovkov08:asymptotic} where such
issues have been thoroughly studied. While we are not 
claiming any significant new results, our goal is to make these results
easier to use for combinatorial applications such as the Random 
Satisfiability problem. By slightly restricting the generality of the 
framework we considerably simplify and `discretize' the proofs.
At the same time we improve the bounds in the inequalities.

More precisely, we consider distributions $X$ with integer values,
whose tail functions $F_{X+}$ (the right tail) and $F_{X-}$
(the left tail) are majorized by power functions from 
$Vx^{-\al_r}$ and $Wx^{-\al_\ell}$, respectively. Note that such
a distribution may have no first moment if $\min(\al_r,\al_\ell) \leq 1$,
and it may have no second moment if $\min(\al_r,\al_\ell) \leq 2$.

We show that if $\min(\al_r,\al_\ell) \leq 1$, then the sum of such 
variables w.h.p.\ does not deviates much from the value one 
may expect (the mean value of such sum does not exists). 
Note that we do not assume that these variables are 
identically distributed.

\begin{theorem}\label{the:<1}
	Let $S_n=\sum_{i=1}^nX_i$, where for each $X_i$ it holds
	$F_{X_i+}(x)\le Vx^{-\al_r}$, for some $0<\al_r \leq 1$. Then for any
	$\ve>0$,
	\[
	\Pr\left[S_n\ge n^{\frac1{\al_r}+\ve}\right]\le(V+e^{2V})n^{-\al_r\ve},
	\]
	when $n\to\infty$.
\end{theorem}

In the case $\min(\al_r,\al_\ell)>1$, the mean of the sum exists
and prove a bound on the probability the deviates  from it
by a certain amount.

\begin{theorem}\label{the:>1}
	Let $S_n=\sum_{i=1}^nX_i$, where for each $X_i$ it holds
	$F_{X_i+}(x)\le Vx^{-\al_r}$, $F_{X_i-}(x)\le Wx^{-\al_\ell}$, 
	for some $\al_r,\al_\ell>1$. Then letting $\al=\min(\al_r,\al_\ell)$, 
	for any $\ve>0$,
	\[
	\Pr\left[S_n-\E S_n\ge n^{\max(1/\al,1/2)+\ve}\right]\le Vn^{1-\max(1,\al/2)-\al\ve}+e^{2V}n^{-\al\ve},
	\]
	when $n\to\infty$.
\end{theorem}

Bounds for the left tail are similar.

The methods we use are fairly standard and boil down to careful 
evaluation of the tails of the sum of the $X_i$'s.

\section{Notation and preliminaries}

We say that some sequence of events $\left\{A_n\right\}$ happens with high probability (\textbf{w.h.p.}), when 
$$
\lim\limits_{n \rightarrow \infty}\Pr[A_n] = 1.
$$

Let $S_n=\sum_{i=1}^{n}X_i$ be the sum of $n$ independent \textit{not necessary} identically distributed integer-valued random variables $X_i$'s with tails that do not depend on $n$ (probably this constraint could be relaxed to some extent, however, in this work we will deal only with variables, which constitute the sum, that do not depend on $n$).

First, we need to introduce a couple of useful concepts that we will exploit heavily in the subsequent chapters. Primarily, we will need \textit{right-} and \textit{left-tail functions}:
\begin{definition}
Let $X$ be some random variable with support on $\mathbb{S} \subseteq \mathbb{R}$. Then the function
$$
F_{X+}(x):=\Pr[X \geq x], \text{ where } x > 0 
$$
is the right-tail function of the r.v. $X$. Similarly, the left-tail function of the r.v. $X$ is
$$
F_{X-}(x):=\Pr[X \leq -x], \text{ where } x > 0.
$$
\end{definition}

Note, though, that the above definition is slightly broader than we need, since it applies to any numerical variable. However, in this paper, we deal with random variables that take integer values, meaning their support $\mathbb{S}=\mathbb{Z}$.

We do not require variables $X_i$'s to be identically distributed, but we do need some property that they all share in order to work with different distributions in a simple and unified way. For that purpose we harness the concept of \textit{majorization} (or domination):
\begin{definition}
We say that a function $g(x)$ majorizes a function $f(x)$, if $g(x) \geq f(x)$ for every $x$ from the domain of $f$. 
\end{definition}

We will focus on variables $X_i$'s with tails that can be majorized by some power functions with negative exponents. As it will be shown later, the values of these exponents play a critical role in the behaviour of the sum $S_n$.   

\begin{definition}
	\label{def:D}
	Random variable $X$  has a probability distribution function that belongs to the set $\mathbb{D}\left(\alpha_l,\alpha_r\right)$, if there exist constants $V>0$ and $W>0$ (which we call left- and right-tail constants of the respective r.v.), such that $F_{X+}(x) \leq V x^{-\alpha_r}$ and $F_{X-}(x) \leq W x^{-\alpha_l}$ for $x > 0$, where $\alpha_l,\alpha_r>0$ are constants (we will call them left and right tail exponents or powers).
	
	The fact that the distribution of a r.v. $X$ is from $\mathbb{D}\left(\alpha_l,\alpha_r\right)$ will be denoted as
	$$
	X_i \sim \mathbb{D}\left(\alpha_l,\alpha_r\right).
	$$
\end{definition}

Although, it may seem that the variables from $\mathbb{D}\left(\alpha_l,\alpha_r\right)$ belong to a rather restrictive class of random variables, however, such majorization can be applied to a very broad range of variables, including gaussian, subexponential, heavy-tailed, and, obviously, power-law random variables. Therefore, all the results of this paper are applicable to these classes of random variables.

We will also say that the variable's distribution is from $\mathbb{D}(\cdot\,,\alpha_r)$ (or $\mathbb{D}(\alpha_l, \, \cdot)$), if the right tail is majorized by some power function with exponent $-\alpha_r$, while the left tail is arbitrary (or, respectively, when the left tail can be majorized by some power function with exponent $-\alpha_l$, while the right tail is arbitrary). Moreover, by $\mathbb{D}(>1, \cdot)$ (or $\mathbb{D}(\cdot, >1)$) we denote the set of distributions with left-tail (or right-tail) functions that can be bounded by some power function $C\,x^{-\alpha}$, where $\alpha>1$. 

Additionally, when finite sequence $\{X_i\}_{i=1}^{n}$ of random variables consists of r.vs., each of which belongs to the class $\mathbb{D}(\, \cdot \,, \, \alpha_{r,i})$, then we use two quantities: 
$$
\alpha_r = \min\left( \alpha_{r,1},\,\alpha_{r,2}, \cdots, \, \alpha_{r,n} \right),
$$
and
$$
V = \max\left( V_1,\,V_2, \cdots, \, V_n \right),
$$
where $V_i$'s are the constants from the majorizing power functions.
Clearly, that each $X_i$'s right tail then can be majorized by a power function  $V\,x^{-\alpha_r}$. In a similar way we define $\alpha_l$ and $W$, that is when all $X_i \sim \mathbb{D}(\alpha_{l,i}, \cdot\,)$, then 
$$
\alpha_l = \min\left( \alpha_{l,1},\,\alpha_{l,2}, \cdots, \, \alpha_{l,n} \right),
$$ 
and
$$
W = \max\left( W_1,\,W_2, \cdots, \, W_n \right).
$$

Additionally, when all $X_i$'s come from the $\mathbb{D}(\alpha_{l,i}, \alpha_{r,i})$ classes , then we use another important quantity $\alpha$, which is the minimum among all $\alpha_{l,i}$'s and $\alpha_{r,i}$'s, i.e.
$$
\alpha = \min\left(\alpha_{l}, \alpha_r\right).
$$

In this paper we show how harnessing the values of tail exponents of variables $X_i$'s allows us to bound probabilities for the sum $S_n=\sum X_i$ to have large values or to deviate much from its expected value (given that its expectation exists). 

We finish the preliminaries part with a simple, yet useful technique, which we use heavily in this paper, that is summation by parts. Although, it is a well-known procedure, however, for the sake of proof completeness, we state it here.

\begin{lemma}[Summation by parts]
\label{lemma:summation-by-parts}	
Let $X$ be some integer-valued random variable. Then for any function $f(x)\geq 0$ the following equalities hold
$$
\sum_{d \leq 0}f(d)\Pr\left[X=d\right] = f(0)\Pr\left[X \leq 0\right]  + \sum_{d \leq -1}\Big(f(d)-f(d+1)\Big)\Pr\left[X_i \leq k\right],
$$
and for any $0 \leq a \leq b < \infty$
$$
\sum_{d = a}^{b}f(d)\Pr\left[X=d\right] = f(a)F_{X+}(a)-f(b)F_{X+}(b)+ \sum_{d =a+1}^{b}\Big(f(d)-f(d-1)\Big) \Pr\left[ X_i \geq k \right].
$$
\end{lemma}
\begin{proof}
Equalities follow	after noticing that $\Pr\left[X = d\right] = \Pr\left[X \leq d\right] - \Pr\left[ X \leq d - 1 \right]$, when $d \leq 0$, and, similarly, $\Pr\left[X = d\right] = \Pr\left[X \geq d\right] - \Pr\left[ X \geq d + 1 \right]$, when $d \geq 0$. Then
\begin{align*}
\sum_{d \leq 0}f(d)\Pr\left[X=d\right] & = \sum_{d \leq 0}f(d)\Big(\Pr\left[X \leq d\right] - \Pr\left[ X \leq d - 1 \right]\Big)\\
& = \sum_{d \leq 0}f(d)\Pr\left[X \leq d\right] -  \sum_{d \leq 0}f(d)\Pr\left[ X \leq d - 1 \right]\\
& = \sum_{d \leq 0}f(d)\Pr\left[X \leq d\right] -  \sum_{d \leq -1}f(d+1)\Pr\left[ X \leq d \right]\\
& = f(0)\Pr\left[X \leq 0\right]  + \sum_{d \leq -1}f(d)\Pr\left[X \leq d\right] -  \sum_{d \leq -1}f(d+1)\Pr\left[ X \leq d \right]\\
& = f(0)\Pr\left[X \leq 0\right]  + \sum_{d \leq -1}\Big(f(d)-f(d+1)\Big)\Pr\left[X \leq d\right].
\end{align*}
As for the sum over positive indices, it goes likewise
\begin{align*}
\sum_{d = a}^{b}f(d)\Pr\left[X=d\right] & = \sum_{d =a}^{b}f(d)\Big(\Pr\left[X \geq d\right] - \Pr\left[ X \geq d + 1 \right]\Big)\\
& = \sum_{d =a}^{b}f(d)\Pr\left[X \geq d\right] -  \sum_{d =a}^{b}f(d)\Pr\left[ X \geq d + 1 \right]\\
& = \sum_{d =a}^{b}f(d)\Pr\left[X \geq d\right] -  \sum_{d =a+1}^{b+1}f(d-1)\Pr\left[ X \geq d \right]\\
& = f(a)\Pr\left[X \geq a\right]-f(b)\Pr\left[ X \geq b+1 \right]+ \sum_{d =a+1}^{b}\Big(f(d)-f(d-1)\Big) \Pr\left[X \geq d\right].
\end{align*}
And the lemma follows.
\end{proof}

\section{Tail bounds of $S_n$, when $0 < \min\left(\alpha_l,\alpha_r\right) \leq 1$}

Tail inequalities in this section are of the most general nature, meaning they can be applied to \textbf{any} variables with tail functions that can be majorized by some power-law function $C\,k^{-\alpha}$ with $\alpha>0$. Unfortunately, there exists a trade-off between the specificity of the inequalities over random variables and the tightness of the bounds they provide, i.e. the more general the inequalities are, the wider bounds on the values of $S_n$ they assert.  

\begin{theorem}
	\label{thr:bound_1}
	Let $S_n=\sum_{i=1}^{n}X_i$, where $X_i \sim \mathbb{D}\left(\:\cdot\:, \alpha_{r,i} \right)$ are independent not necessary identically distributed integer-valued random variables, and $0< \alpha_r \leq 1$.  Then for any $\epsilon > 0$, we have
	\begin{align}
		\Pr[S_{n} \geq n^{\frac{1}{\alpha_r}+\epsilon}] & \leq \left(V+e^{2V}\right) n^{-\alpha_r\epsilon},
	\end{align}
	when $n \rightarrow \infty$.
\end{theorem}

%\begin{remark}
%	As we mentioned earlier, if r.v. $X \sim \mathbb{D}\left(\:\cdot\:,>1\right)$, then $X \sim \mathbb{D}\left(\:\cdot\:, \alpha_{r}\right)$, where $0<\alpha_{r} \leq 1$. Thus, if some $X_i$'s in the sum $S_n$  belong to the class $\mathbb{D}\left(\:\cdot\:,\alpha_{r,i}\right)$ with $\alpha_{r,i} > 1$, then for the sake of uniformity we treat them as variables with right tails that can be bounded by power function with exponent that is less than or equal 1. Such approach makes the following calculations ''cleaner``. 
	
%	However, if one needs to obtain tighter numerical bounds while dealing with the sum of independent random variables each of which belongs either to  $\mathbb{D}\left(\:\cdot\:,>1\right)$ or $\mathbb{D}\left(\:\cdot\:, \leq 1\right)$, then it shouldn't be hard to adapt calculations from this and the next section to get the desired inequality. However, such bound will be rather messy. 
%\end{remark}	

\begin{proof}
	First, as was mentioned previously, when all $X_i \sim X_i \sim \mathbb{D}\left(\:\cdot\:, \alpha_{r,i} \right)$, then
	\begin{align}
		F_{X_i+}(k) \leq Vk^{-\alpha_r},  \label{eq:right-tail-inequality}
	\end{align}
	which is valid for all  $k \geq 1$ and any $1 \leq i \leq n$.
	
	Next, let us introduce events $B_i:=\{X_i \leq x\}$, which indicate that the variable $X_i$ does not exceed $x$, and also the event $B$ that none of the variables $X_i$'s exceed $x$, i.e.
    $$
    B := \bigcap_{i=1}^{n}B_i.
    $$
    Then we have
    \begin{align}
    \Pr\left[S_n \geq x \right] & = \Pr\left[\{ S_n \geq x \} \wedge \bar{B}\right] + \Pr\left[\{ S_n \geq x \} \wedge B\right]\nonumber \\
    & \leq \Pr\left[\bar{B}\right] + \Pr\left[\{ S_n \geq x \} \wedge B\right]\nonumber \\
    & \leq \Pr\left[\bigcup_{i=1}^{n}\bar{B}_i\right] + \Pr\left[ S_n \geq x \: | \: B\right]\, \cdot \Pr\left[B\right]. \label{eq:basic-inequality}
    \end{align}
    
    Next, 
    
    \begin{align}
    \Pr\left[\bigcup_{i=1}^{n}\bar{B}_i\right] & = \Pr\left[\bigcup_{i=1}^{n}  \{ X_i > x  \} \right] \nonumber\\
        & \leq \sum_{i=1}^{n}\Pr\left[ X_i \geq x  \right] \text{ (by Union bound)}\nonumber\\
        & = \sum_{i=1}^{n}F_{X+}(x)\nonumber\\
        & \leq \sum_{i=1}^{n}Vx^{-\alpha_r},\: \left(\text{from~\eqref{eq:right-tail-inequality}}\right)\nonumber \\
        & = nVx^{-\alpha_r} \label{eq:pr_not_b},
    \end{align}
    
    and since $X_i$'s are independent random variables, we also have
    \begin{align}
    \Pr[B] = \Pr\left[ \bigcap_{i=1}^{n}B_i \right] = \prod_{i=1}^{n}\Pr\left[ X_i \leq x \right]. \label{eq:pr_b}
    \end{align}	
   Therefore, inequality~\eqref{eq:basic-inequality} can be simplified further to
   \begin{align}
    \Pr\left[S_n \geq x \right] & \leq \Pr\left[\bigcup_{i=1}^{n}\bar{B}_i\right] + \Pr\left[ S_n \geq x \: | \: B\right]\, \cdot \Pr\left[B\right] \nonumber \\
    & \leq nVx^{-\alpha_r} + \Pr\left[ S_n \geq x \: | \: B\right]\, \cdot \prod_{i=1}^{n}\Pr\left[ X_i \leq x \right]. \label{eq:basic-inequality-2}
   \end{align}
   Now we take a closer look at the probability $\Pr[S_{n} \geq x \: | \: B]$. This sum consists of $n$ independent random variables having conditionally the same distribution as the original random variables $X_i$'s but truncated at $x$. More formally, let us introduce new truncated random variables $Y_i$ with the following probability distribution function
   $$
   \Pr[Y_i = k] = \begin{cases}
   \frac{\Pr[X_i= k]}{\Pr[X_i \leq x] }, & \text{when } k \leq x \\
   0, & \text{otherwise.}
   \end{cases}
   $$  
   Then
   $$
   \Pr\left[S_{n} \geq x \: | \: B \right] = \Pr\left[\sum_{i=1}^{n}Y_i \geq x\right] =: \Pr\left[ S_n^{\langle x \rangle} \geq x\right],
   $$
   where $S_n^{\langle x \rangle}$ is the sum of the $Y_i$'s. After introducing $Y_i$'s and $S_n^{\langle x \rangle}$, we can rewrite~\eqref{eq:basic-inequality-2} as
   \begin{align}
   \Pr\left[S_n \geq x \right] & \leq nVx^{-\alpha_r} + \Pr\left[ S_n \geq x \: | \: B\right]\, \cdot \prod_{i=1}^{n}\Pr\left[ X_i \leq x \right] \nonumber \\
   & \leq nVx^{-\alpha_r} + \Pr\left[ S_n^{\langle x \rangle} \geq x\right]\, \cdot \prod_{i=1}^{n}\Pr\left[ X_i \leq x \right] \label{eq:final-inequality-3}. 
   \end{align} 
   
   Next, from Markov's inequality we know that for any non-negative random variable $A$, positive $x$ and non-negative non-decreasing function $\phi(x)$ we have
   $$
   \Pr[A \geq x] \leq \frac{\E\,\phi(A)}{\phi(x)}.
   $$
   Applying this inequality to the $ S_n^{\langle x \rangle}$ and letting $\phi(x)=e^{\mu x}$ for some positive $\mu$, which will be defined later (in what follows, we will maintain $\mu \rightarrow 0$ when $n \rightarrow \infty$, but $\mu x \rightarrow \infty$. Moreover, we assume that $n$ is large enough for $\mu \leq 1$), we obtain
   \begin{align*}
   \Pr[S_{n}^{ \langle x \rangle} \geq  x] \leq \frac{\E e^{\mu S_{n}^{ \langle x \rangle}}}{e^{\mu x}}=e^{-\mu x}\E e^{\mu \sum_{i=1}^{n}Y_{i}}=e^{-\mu x}\E \prod_{i=1}^{n} e^{\mu Y_{i}}.
   \end{align*}
   Exploiting the fact that $Y_{i}$'s are independent random variables, we can further simplify the above probability:
   \begin{align*}
   \Pr[S_{n}^{ \langle x \rangle} \geq  x] & \leq e^{-\mu x}\E \prod_{i=1}^{n} e^{\mu Y_{i}} \\
   & = e^{-\mu x}\prod_{i=1}^{n}\E e^{\mu Y_{i}}\\
   & = e^{-\mu x}\prod_{i=1}^{n} \sum_{k \leq x} e^{\mu k}\Pr[ Y_i=k] \\
   & = e^{-\mu x}\prod_{i=1}^{n}\sum_{k \leq x} e^{\mu k}\frac{\Pr[X_i=k]}{\Pr[X_i \leq x ]} \\
   & = e^{-\mu x} \frac{\prod_{i=1}^{n}\sum_{k \leq x} e^{\mu k}\Pr[X_i=k]}{\prod_{i=1}^{n}\Pr[X_i \leq x ]}.
   \end{align*}
   Plugging this inequality for $\Pr[S_{n}^{ \langle x \rangle} \geq  x]$ back into~\eqref{eq:final-inequality-3}, we get
   \begin{align*}
   \Pr[S_{n} \geq x] & \leq nVx^{-\alpha_r} + \Pr\left[ S_n^{\langle x \rangle} \geq x\right]\, \cdot \prod_{i=1}^{n}\Pr\left[ X_i \leq x \right] \\
   & \leq nVx^{-\alpha_r} + e^{-\mu x} \frac{\prod_{i=1}^{n}\sum_{k \leq x} e^{\mu k}\Pr[X_i=k]}{\prod_{i=1}^{n}\Pr[X_i \leq x ]}\, \cdot \prod_{i=1}^{n}\Pr\left[ X_i \leq x \right] \\
   & = nVx^{-\alpha_r} + e^{-\mu x} \prod_{i=1}^{n}\sum_{k \leq x} e^{\mu k}\Pr[X_i=k]\\
   & =: nVx^{-\alpha_r} + e^{-\mu x} \prod_{i=1}^{n}R_i(\mu,x),
   \end{align*}
   where $R_i(\mu, x):=\sum_{k \leq x} e^{\mu k}\Pr[X_i=k]$. Clearly, in order to obtain the final bound, we need to upper bound $R_i(\mu, x)$. For that purpose, we split sum in $R_i(\mu, x)$ into 3 disjoint intervals, and bound each interval separately:
   \begin{align}
   R_i(\mu, x) 	& = \sum_{k \leq x} e^{\mu k}\Pr[X_i=k] \nonumber \\
   & \leq  \Big[\sum_{k \leq 0}e^{\mu k}\Pr[X_i=k]\Big]+\Big[\sum_{k=1}^{\lfloor M \rfloor}e^{\mu k}\Pr[X_i=k]\Big]
   + \Big[\sum_{k=\lceil M \rceil}^{x}e^{\mu k}\Pr[X_i=k]\Big] \nonumber \\
   & =: I_{i,1} + I_{i,2} + I_{i,3},\nonumber
   \end{align}
   where $M=\frac{2\alpha_r}{\mu}$ (now it should be clear why we require $\mu x \rightarrow \infty$, since we would like to have  $0 \ll M < x$). 
   
  \subsection{Upper bound of $I_{i,1}$, when $0 < \alpha_r \leq 1$}
  
  The first interval is the easiest to bound, as the next lemma shows this.
  
  \begin{lemma}
    \label{lemma:Ii1}
    Let $X_i \sim \mathbb{D}(\:\cdot\:, \alpha_{r})$ be an integer valued r.v. with $0 < \alpha_r \leq 1$ and an arbitrary left tail function. Then 
    $$
    I_{i,1} = \sum_{k \leq 0}e^{\mu k}\Pr[X_i=k]
    $$ 
    is bounded from above by
    $$
    I_{i,1} \leq \Pr\left[X_i \leq 0\right].
    $$
  \end{lemma}
  \begin{proof}
  
  The proof is straightforward. For $I_{i,1}$ we have
  \begin{align*}
  I_{i,1} & = \sum_{k \leq 0}e^{\mu k}\Pr[X_i=k] \leq \sum_{k \leq 0}\Pr[X_i=k] = \Pr\left[X_i \leq 0\right].
  \end{align*}
  
  \end{proof}
  
  \subsection{Upper bound of $I_{i,2}$, when $0 < \alpha_r \leq 1$}
  
  In order to bound the second interval, we apply a slightly more sophisticated approach than that for $I_{i,1}$. Here we first apply summation by parts to express $I_{i,2}$ as a function of the right tail $F_{X_i+}(k)$, and then we use the integral bound of summation to upper bound the $I_{i,2}$ interval. 
  
  \begin{lemma}
  \label{lemma:Ii2}
  Let $X_i \sim \mathbb{D}(\:\cdot\:, \alpha_{r})$ be an integer valued r.v. with $0 < \alpha_r \leq 1$ and an arbitrary left-tail function. Then 
  $$
  I_{i,2} = \sum_{k=1}^{\lfloor M \rfloor}e^{\mu k}\Pr[X_i=k],
  $$ 
  where $M=\frac{2\alpha_r}{\mu}$ and $0 < \mu < 1$, is bounded from above by
  $$
  I_{i,2} \leq \Pr[X_i \geq 1]  + \begin{cases}
  O\left(\mu^{\alpha_r}\right), & \text{ when } 0< \alpha_r < 1,\\
  O\left(- \mu\ln\mu\right), & \text{ when } \alpha_r = 1.
  \end{cases}
  $$
\end{lemma}
\begin{proof}
As was mentioned, we first apply summation by parts (Lemma~\ref{lemma:summation-by-parts}) to $I_{i,2}$. Hence, we obtain
\begin{align*}
I_{i,2} 	& = \sum_{k=1}^{\lfloor M \rfloor}e^{\mu k}\Pr[X_i=k] \\
& =  e^{\mu}\Pr[X_i \geq 1] - e^{\mu \lfloor M \rfloor}\Pr\left[X_i \geq \lfloor M \rfloor\right] + \sum_{k=2}^{\lfloor M \rfloor}\left( e^{\mu k} - e^{\mu (k-1)} \right)\Pr[X_i \geq k] \\
& \leq  e^{\mu}\Pr[X_i \geq 1] + \sum_{k=2}^{\lfloor M \rfloor}\left( e^{\mu k} - e^{\mu (k-1)} \right)\Pr[X_i \geq k] \\
& \leq  e^{\mu}\Pr[X_i \geq 1] + \sum_{k=2}^{\lfloor M \rfloor}\left( 1 - e^{-\mu} \right)e^{\mu k}\Pr[X_i \geq k].
\end{align*}
Since $1 -\mu \leq e^{-\mu}$, we have 
\begin{align*}
I_{i,2} 	& \leq  e^{\mu}\Pr[X_i \geq 1] + \sum_{k=2}^{\lfloor M \rfloor}\left( 1 - e^{-\mu} \right)e^{\mu k}\Pr[X_i \geq k]\\
& \leq e^{\mu}\Pr[X_i \geq 1] + \mu\sum_{k=2}^{\lfloor M \rfloor}e^{\mu k}\Pr[X_i \geq k].
\end{align*}
Moreover, when $\mu$ is close to 0, then $e^{\mu} \leq 1 + 2\mu$, and therefore,
\begin{align*}
I_{i,2} 	& \leq e^{\mu}\Pr[X_i \geq 1] + \mu\sum_{k=2}^{\lfloor M \rfloor}e^{\mu k}\Pr[X_i \geq k]\\
& \leq \Pr[X_i \geq 1] + 2\mu\Pr[X_i \geq 1] + \mu\sum_{k=2}^{\lfloor M \rfloor}e^{\mu k}\Pr[X_i \geq k]\\
& \leq \Pr[X_i \geq 1] + 2\mu + \mu\sum_{k=2}^{\lfloor M \rfloor}e^{\mu k}\Pr[X_i \geq k].
\end{align*} 

Now recall that $Pr\left[ X_i \geq k \right] = F_{X_i+}(k) \leq Vk^{-\alpha_r}$ with $0 < \alpha_r \leq 1$. Then 

\begin{align*}
I_{i,2} & \leq \Pr[X_i \geq 1] + 2\mu + \mu\sum_{k=2}^{\lfloor M \rfloor}e^{\mu k}\Pr[X_i \geq k]\\
& \leq\Pr[X_i \geq 1] + 2\mu + V\mu \sum_{k=2}^{\lfloor M \rfloor}e^{\mu k}k^{-\alpha_r}\\
& \leq \Pr[X_i \geq 1] + 2\mu + Ve^{\mu \lfloor M \rfloor}\mu  \sum_{k=2}^{\lfloor M \rfloor}k^{-\alpha_r}\\
& \leq \Pr[X_i \geq 1] + 2\mu + Ve^{2\alpha_r}\mu  \sum_{k=2}^{\lfloor M \rfloor}k^{-\alpha_r},
\end{align*}
since $M= \frac{2\alpha_r}{\mu}$. 

Function $k^{-\alpha_r}$ is monotonically decreasing when $\alpha_r>0$, hence, we can apply the integral upper bound, i.e. for some decreasing in $[a \dots b]$ function $\psi(k)$, we have

\begin{align}
\sum_{k=a}^{b}\psi(k) \leq \int\limits_{a-1}^{b}\psi(t)\diff t.\label{eq:integral-bound}
\end{align}

Then  

\begin{align}
I_{i,2} & \leq \Pr[X_i \geq 1] + 2\mu + Ve^{2\alpha_r}\mu  \sum_{k=2}^{\lfloor M \rfloor}k^{-\alpha_r} \nonumber \\
& \leq \Pr[X_i \geq 1] + 2\mu + Ve^{2\alpha_r}\mu \int\limits_{1}^{M }t^{-\alpha_r} \diff t \nonumber \\
& =: \Pr[X_i \geq 1] + 2\mu + Ve^{2\alpha_r} \,I_{i,2}^0 \label{eq:integral_2},
\end{align}
where $I_{i,2}^0=\mu\int\limits_{1}^{M}t^{-\alpha_r} \diff t$. Next, let us bound the $I_{i,2}^0$ term (and recall that $M = \frac{2\alpha_r}{\mu}$):
\begin{align*}
I_{i,2}^0&=\mu\int\limits_{1}^{M}t^{-\alpha_r} \diff t\\
& = \mu \begin{cases}
 \frac{ M^{1-\alpha_r} }{1-\alpha_r}-\frac{1 }{1-\alpha_r}, & \text{ when } 0< \alpha_r < 1,\\
 \ln M, & \text{ when } \alpha_r = 1
\end{cases}\\
& \leq \begin{cases}
 \frac{\left(2\alpha_r\right)^{1-\alpha_r}}{1-\alpha_r}\mu^{\alpha_r}, & \text{ when } 0< \alpha_r < 1,\\
 \mu\ln 2 - \mu\ln\mu, & \text{ when } \alpha_r = 1
\end{cases}\\
& \leq \begin{cases}
O\left(\mu^{\alpha_r}\right), & \text{ when } 0< \alpha_r < 1,\\
O\left(- \mu\ln\mu\right), & \text{ when } \alpha_r = 1.
\end{cases}
\end{align*}

Thus, from~\eqref{eq:integral_2} we obtain
\begin{align*}
I_{i,2} & \leq \Pr[X_i \geq 1] + 2\mu + Ve^{2\alpha_r} \,I_{i,2}^0\\
& \leq \Pr[X_i \geq 1] + 2\mu + \begin{cases}
O\left(\mu^{\alpha_r}\right), & \text{ when } 0< \alpha_r < 1,\\
O\left(- \mu\ln\mu\right), & \text{ when } \alpha_r = 1,
\end{cases}\\
&= \Pr[X_i \geq 1]  + \begin{cases}
O\left(\mu^{\alpha_r}\right), & \text{ when } 0< \alpha_r < 1,\\
O\left(- \mu\ln\mu\right), & \text{ when } \alpha_r = 1,
\end{cases}
\end{align*}
where the last equality follows from the assumption that $0< \mu <1$ is some number close to 0; hence, the lemma is proved.
\end{proof}

\subsection{Upper bound of $I_{i,3}$}

Proof for this interval's bound closely resembles the proof of the bound of the second interval $I_{i,2}$ differing only in some details. However, unlike the previous interval, the bound for $I_{i,3}$ presented below is valid for any $\alpha_r>0$. We emphasize this observation, since result of Lemma~\ref{lemma:Ii3} will be re-used to prove a bound for a similar interval when $\alpha_r > 1$ (see Lemma~\ref{lemma:ji3}). 

\begin{lemma}
	\label{lemma:Ii3}
	Let $X_i \sim \mathbb{D}(\:\cdot\:, \alpha_{r})$ be an integer valued r.v. with $ \alpha_{r} > 0$ and an arbitrary left tail  function. Then 
	$$
	I_{i,3} = \sum_{k={\lceil M \rceil}}^{x}e^{\mu k}\Pr[X_i=k],
	$$ 
	where $M=\frac{2\alpha_r}{\mu}$ and $0 < \mu < 1$,	is bounded from above by
	$$
	I_{i,3} \leq O\left(\mu^{\alpha_r}\right) + Ve^{\mu x}x^{-\alpha_r}.
	$$
\end{lemma}
\begin{proof}
First, just like with the previous interval $I_{i,2}$, we apply summation by parts(Lemma~\ref{lemma:summation-by-parts}) :
\begin{align*}
I_{i,3} & = \sum_{k={\lceil M \rceil}}^{x}e^{\mu k}\Pr[X_i=k]\\
& = \sum_{k={\lceil M \rceil}}^{x}e^{\mu k}\Big(\Pr[X_i \leq k] - \Pr[X_i \leq k-1]\Big)\\
& = e^{\mu\lceil M \rceil}\Pr\left[ X_i\geq \lceil M \rceil \right] - e^{\mu x}\Pr\left[ X_i\geq x \right] +  \sum_{k=\lceil M \rceil + 1}^{x}\left(e^{\mu k} - e^{\mu(k-1)}\right)\Pr[X_i \geq k]\\
& = e^{\mu\lceil M \rceil}\Pr\left[ X_i\geq \lceil M \rceil \right] - e^{\mu x}\Pr\left[ X_i\geq x \right] +  \sum_{k=\lceil M \rceil}^{x-1}\left(e^{\mu (k+1)} - e^{\mu k}\right)\Pr[X_i \geq k+1]\\
& \leq e^{\mu\lceil M \rceil}\Pr\left[ X_i\geq \lceil M \rceil \right] +  \sum_{k=\lceil M \rceil}^{x-1}\left(e^{\mu} - 1\right)e^{\mu k}\Pr[X_i \geq k].
\end{align*}

Again, recall that $ e^{\mu} \leq 1 +2\mu$, when $\mu>0$ is close to 0. Hence, we obtain
\begin{align*}
I_{i,3} & \leq e^{\mu\lceil M \rceil}\Pr\left[ X_i\geq \lceil M \rceil \right] +  \sum_{k=\lceil M \rceil}^{x-1}\left(e^{\mu} - 1\right)e^{\mu k}\Pr[X_i \geq k]\\
& \leq e^{\mu\lceil M \rceil}\Pr\left[ X_i\geq \lceil M \rceil \right] +  2\mu\sum_{k=\lceil M \rceil}^{x-1}e^{\mu k}\Pr[X_i \geq k]\\
\end{align*}

Next, since $\Pr\left[X_i \geq k\right]=:F_{X_i+}(k) \leq V\,k^{-\alpha_r}$ when $k \geq 1$, we have
\begin{align*}
I_{i,3} & \leq e^{\mu\lceil M \rceil}\Pr\left[ X_i\geq \lceil M \rceil \right] +  2\mu\sum_{k=\lceil M \rceil }^{x-1}e^{\mu k}\Pr[X_i \geq k]\\
& \leq e^{\mu\lceil M \rceil}V\,M^{-\alpha_r}  +  2V\mu\sum_{k=\lceil M \rceil }^{x-1}e^{\mu k}k^{-\alpha_r}\\
& \leq V\frac{e^{3\alpha_r}}{(2\alpha_r)^{\alpha_r}}\,\mu^{\alpha_r}  +  2V\mu\sum_{k=\lceil M \rceil }^{x-1}e^{\mu k}k^{-\alpha_r}, \text{ since } M=\frac{2\alpha_r}{\mu}\\
& = O\left(\mu^{\alpha_r}\right)  +  2V\mu\sum_{k=\lceil M \rceil }^{x-1}e^{\mu k}k^{-\alpha_r}.
\end{align*}
Next, let's investigate the monotonicity of the function under summation. We have
$$
\frac{\diff}{\diff k}\Big[e^{\mu k}k^{-\alpha_r}\Big]=\mu e^{\mu k}k^{-\alpha_r}-\alpha_r e^{\mu k}k^{-\alpha_r-1}=\mu e^{\mu k}k^{-\alpha_r-1}\left(k-\frac{\alpha_r}{\mu}\right),
$$
which is clearly positive when $k \geq 2\alpha_r/\mu = M$. Thus, when $k \geq M$, the function $e^{\mu k}k^{-\alpha_r}$ is monotonically increasing. Therefore, we can apply the integral upper bound of summation of an increasing function, i.e. for some increasing in $[a \dots b]$ function $\psi(k)$, we have
$$
\sum_{k=a}^{b}\psi(k) \leq \int\limits_{a}^{b+1}\psi(t)\diff t.
$$
So
\begin{align*}
I_{i,3} & \leq O\left(\mu^{\alpha_r}\right)  +  2V\mu\sum_{k=\lceil M \rceil }^{x-1}e^{\mu k}k^{-\alpha_r}\\
& \leq O\left(\mu^{\alpha_r}\right) + 2V\mu\int\limits_{M}^{x}e^{\mu t}t^{-\alpha_r}\diff t.
\end{align*}
Using substitution of variables $u:=\mu(x-t)$ and $\diff t = -\frac{1}{\mu}\diff u$, the above integral can be transformed into
\begin{align}
I_{i,3} & \leq O\left(\mu^{\alpha_r}\right) + 2V\mu\int\limits_{M}^{x}e^{\mu t}t^{-\alpha_r}\diff t \nonumber \\
& = O\left(\mu^{\alpha_r}\right) - 2V\int\limits_{\mu(x-M)}^{0}e^{\mu x - u  }\Big(x-\frac{u}{\mu}\Big)^{-\alpha_r}\diff u \nonumber \\
& = O\left(\mu^{\alpha_r}\right) + 2Ve^{\mu x}x^{-\alpha_r}\int\limits_{0}^{\mu(x-M)}e^{ - u  }\Big(1-\frac{u}{\mu x}\Big)^{-\alpha_r}\diff u. \label{eq:int}
\end{align}
Now consider the function $f(u)=\Big(1-\frac{u}{\mu x}\Big)^{-\alpha_r}$. We have for $u \in \left[0\cdots\mu(x-M)\right]$
\begin{align*}
\frac{\diff}{\diff u}\ln f(u)& =\frac{\diff}{\diff u}\ln \Big(1-\frac{u}{\mu x}\Big)^{-\alpha_r}\\
& = -\alpha_r\frac{\diff}{\diff u}\ln\Big(1-\frac{u}{\mu x}\Big) \\
& = \frac{\alpha_r}{\mu x - u}\\
& \leq \frac{1}{2},
\end{align*}
where the last inequality follows from the fact that function $\frac{\alpha_r}{\mu x - u}$ reaches its maximum at the rightmost point when $u = \mu(x-M)=\mu x - 2\alpha_r$. Furthermore, $f(0)=1$, hence, $f(u) \leq e^{u/2}$, when $0 \leq u \leq \mu(x-M)$.  

Then~\eqref{eq:int} can be upper bounded by
\begin{align*}
I_{i,3} & = O\left(\mu^{\alpha_r}\right) + 2Ve^{\mu x}x^{-\alpha_r}\int\limits_{0}^{\mu(x-M)}e^{ - u  }\Big(1-\frac{u}{\mu x}\Big)^{-\alpha_r}\diff u\\
& = O\left(\mu^{\alpha_r}\right) + 2Ve^{\mu x}x^{-\alpha_r}\int\limits_{0}^{\mu(x-M)}e^{ - u  }f(u)\diff u\\
& \leq O\left(\mu^{\alpha_r}\right) + 2Ve^{\mu x}x^{-\alpha_r}\int\limits_{0}^{\mu(x-M)}e^{ - u  }e^{u/2}\diff u\\
%& = O\left(\mu^{\alpha_r}\right) + 2Ve^{\mu x}x^{-\alpha_r}\int\limits_{0}^{\mu(x-M)}e^{ - u/2  }\diff u\\
& \leq O\left(\mu^{\alpha_r}\right) + 2Ve^{\mu x}x^{-\alpha_r}\int\limits_{0}^{\infty}e^{ - u/2  }\diff u\\
& = O\left(\mu^{\alpha_r}\right) + Ve^{\mu x}x^{-\alpha_r},
\end{align*} 
and, therefore, the lemma is proved.
\end{proof}

\subsection{Final assembling steps, when $0 < \alpha_r \leq 1$}

Recall, that the goal  was to bound
\begin{align}
   \Pr[S_{n} \geq x] & \leq nVx^{-\alpha_r} + e^{-\mu x} \prod_{i=1}^{n}R_i(\mu,x)\label{eq:sum_prob},
\end{align}
where $x = n^{\frac{1}{\alpha_r}+\epsilon}$, and
$$
R_i(\mu, x) \leq I_{i,1}+I_{i,2}+I_{i,3}.
$$  

After proving Lemmas~\ref{lemma:Ii1},~\ref{lemma:Ii2}, and~\ref{lemma:Ii3}, we can obtain a final bound of $\Pr[S_{n} \geq x]$. As was shown,
\begin{align*}
I_{i,1} & \leq \Pr[X_i \leq 0],\\
I_{i,2} & \leq \Pr[X_i \geq 1]  + \begin{cases}
  O\left(\mu^{\alpha_r}\right), & \text{ when } 0< \alpha_r < 1,\\
  O\left(- \mu\ln\mu\right), & \text{ when } \alpha_r = 1,
  \end{cases}\\
I_{i,3} & \leq O\left(\mu^{\alpha_r}\right) + Ve^{\mu x}x^{-\alpha_r}.
\end{align*}

Hence,
\begin{align}
R_i(\mu, x) & =I_{i,1}+I_{i,2}+I_{i,3} \nonumber \\
& \leq \Pr[X_i \leq 0]  + \Pr[X_i \geq 1] +  O\left(\mu^{\alpha_r}\right) + Ve^{\mu x}x^{-\alpha_r} + \begin{cases}
  O\left(\mu^{\alpha_r}\right), & \text{ when } 0< \alpha_r < 1,\\
  O\left(- \mu\ln\mu\right), & \text{ when } \alpha_r = 1,
  \end{cases} \nonumber \\
& = 1 + Ve^{\mu x}x^{-\alpha_r} + \begin{cases}
  O\left(\mu^{\alpha_r}\right), & \text{ when } 0< \alpha_r < 1,\\
  O\left(- \mu\ln\mu\right), & \text{ when } \alpha_r = 1.
  \end{cases} \nonumber \\
& =: 1 + Ve^{\mu x}x^{-\alpha_r} + T_{0}(\alpha_r,\mu)   \label{eq:R_i},
\end{align}
where 
$$
T_0(\alpha_r,\mu) = \begin{cases}
  O\left(\mu^{\alpha_r}\right), & \text{ when } 0< \alpha_r < 1,\\
  O\left(- \mu\ln\mu\right), & \text{ when } \alpha_r = 1.
  \end{cases}
$$

Since $1+x \leq e^{x}$ for every $x$, we further have from~\eqref{eq:R_i}

\begin{align*}
R_i(\mu, x) & \leq  1 + T_0(\alpha_r, \mu) + Ve^{\mu x}x^{-\alpha_r}\\
& \leq \exp\Big( T_0(\alpha_r, \mu) + Ve^{\mu x}x^{-\alpha_r} \Big),  
\end{align*}
and then~\eqref{eq:sum_prob} transforms into
\begin{align}
 \Pr[S_{n} \geq x] & \leq nVx^{-\alpha_r} + e^{-\mu x} \prod_{i=1}^{n}R_i(\mu,x) \nonumber \\
 & \leq nVx^{-\alpha_r} + e^{-\mu x} \prod_{i=1}^{n}\exp\Big( T_0(\alpha_r, \mu) + Ve^{\mu x}x^{-\alpha_r}  \Big) \nonumber \\
  & = nVx^{-\alpha_r} + e^{-\mu x}\exp\Big( nT_0(\alpha_r, \mu) + Ve^{\mu x}nx^{-\alpha_r}  \Big) \nonumber \\
 & = nVx^{-\alpha_r} + \exp\Big(-\mu x + nT_0(\alpha_r, \mu) + Ve^{\mu x}nx^{-\alpha_r}  \Big)  \label{eq:p_u}.
\end{align} 

Next, we need to fix the value of $\mu$, such that the above exponent is minimized as much as possible, while keeping $\mu \rightarrow 0$ but $\mu x \rightarrow \infty$ to make sure that $M=2\alpha_r/\mu$ is much greater than 0, yet less than $x$. One such possible value is $\mu = \frac{1}{x}\ln\frac{x^{\alpha_r}}{n}$. The next lemma verifies that the chosen value of $\mu$ satisfies both constraints, when $0<\alpha_r\leq 1$.

\begin{lemma}
Let $x=n^{\frac{1}{\alpha_r}+\epsilon}$ with $0<\alpha_r \leq 1$ and any $\epsilon>0$. Then
$$
\mu=\frac{1}{x}\ln\frac{x^{\alpha_r}}{n} \rightarrow 0,
$$
but
$$
\mu x = \ln\frac{x^{\alpha_r}}{n} \rightarrow \infty,
$$	
when $n \rightarrow \infty$.
\end{lemma}

\begin{proof}
	Simple calculation shows that
	\begin{align*}
	\mu & = \frac{1}{x}\ln\frac{x^{\alpha_r}}{n}\\
	& = x^{-1}\left(\alpha_r\ln x-\ln n\right)\\
	& = n^{-\frac{1}{\alpha_r}-\epsilon}\left(\alpha_r\ln n^{\frac{1}{\alpha_r}+\epsilon}-\ln n\right)\\
	& = n^{-\frac{1}{\alpha_r}-\epsilon}\left((1+\alpha_r\epsilon)\ln n-\ln n\right)\\
	& \leq \alpha_r\epsilon n^{-\frac{1}{\alpha_r}-\epsilon}\ln n\\
	& = o(1). 
	\end{align*}

In a similar way we prove that $\mu x \rightarrow \infty$.
	\begin{align*}
	\mu x & = \frac{x}{x}\ln\frac{x^{\alpha_r}}{n}\\
	& = \alpha_r\ln x-\ln n\\
	& = \alpha_r\ln n^{\frac{1}{\alpha_r}+\epsilon}-\ln n\\
	& = (1+\alpha_r\epsilon)\ln n-\ln n\\
	& = \alpha_r\epsilon \ln n \rightarrow \infty, 
	\end{align*}
	when $n \rightarrow \infty$.
\end{proof}

Next, after fixing $\mu$, we analyze the exponent in~\eqref{eq:p_u}
\begin{align*}
 \Pr[S_{n} \geq x] & \leq nVx^{-\alpha_r} + \exp\Big(-\mu x + nT_0(\alpha_r, \mu) + Ve^{\mu x}nx^{-\alpha_r}  \Big)\\
 & = nVx^{-\alpha_r} + \exp\Big( nT_0(\alpha_r, \mu) + \left(Ve^{\mu x}nx^{-\alpha_r}-\mu x\right)  \Big)
\end{align*}
by studying asymptotic behaviour of its two components, i.e. $nT_0(\alpha_r, \mu) $ and $Ve^{\mu x}nx^{-\alpha_r}-\mu x$. 

\begin{lemma}
	\label{lemma:W_1}
	Let $x=n^{\frac{1}{\alpha_r}+\epsilon}$ and $\mu=\frac{1}{x}\ln\frac{x^{\alpha_r}}{n}$ with $0<\alpha_r \leq 1$ and any $\epsilon>0$. Then
	$$
	 nT_0(\alpha_r, \mu) = o(1),
	$$	
	where 
	$$
	T_0(\alpha_r, \mu)  = \begin{cases}
	O\left(\mu^{\alpha_r}\right), & \text{ when } 0< \alpha_r < 1,\\
	O\left(- \mu\ln \mu\right), & \text{ when } \alpha_r = 1.
	\end{cases}
	$$
\end{lemma}
\begin{proof}
	Let us consider cases. When $0 < \alpha_r < 1$, we have $T_0(\alpha_r, \mu)  =O\left(\mu^{\alpha_r}\right)$, and so
	\begin{align*}
	nT_0(\alpha_r, \mu) & = O\left(n\mu^{\alpha_r}\right)\\
	& = O\left(\frac{n}{x^{\alpha_r}}\ln^{\alpha_r}\frac{x^{\alpha_r}}{n}\right)\\
	& = O\left(\frac{n}{n^{1+\alpha_r\epsilon}}\ln^{\alpha_r}\frac{n^{1+\alpha_r\epsilon}}{n}\right)\\
	& = O\left(n^{-\alpha_r\epsilon} \ln^{\alpha_r}n^{\alpha_r\epsilon} \right)\\
	& = o(1).
	\end{align*}
	However, when $\alpha_r=1$, then $T_0(\alpha_r, \mu)=O\left(\mu\ln \mu\right)$, while
	\begin{align*}
	nT_0(\alpha_r, \mu) &= O\left(-n\mu\ln \mu\right)\\
	 &= O\left(-\frac{n}{x}\ln\frac{x^{\alpha_r}}{n} \cdot\ln \left(\frac{1}{x}\ln\frac{x^{\alpha_r}}{n}\right)\right) \\
	 &= O\left(-\frac{n}{n^{1/\alpha_r+\epsilon}}\ln\frac{n^{1+\alpha_r\epsilon}}{n} \cdot\ln \left(\frac{1}{n^{1/\alpha_r+\epsilon}}\ln\frac{n^{1+\alpha_r\epsilon}}{n} \right)\right)\\
	 &= O\left(-\frac{n}{n^{1+\epsilon}}\ln\frac{n^{1+\epsilon}}{n} \cdot\ln\left( \frac{1}{n^{1+\epsilon}}\ln\frac{n^{1+\epsilon}}{n}\right)\right), \text{ since } \alpha_r=1 \\
	 &= O\left(-\epsilon n^{-\epsilon}\ln n \cdot\ln\left( \epsilon n^{-(1+\epsilon)}\ln n\right)\right)\\
	  &= O\left(\epsilon(1+\epsilon) n^{-\epsilon}\ln^2 n - \epsilon n^{-\epsilon}\ln n\ln\ln n^{\epsilon}\right)\\
	 &= O\left( n^{-\epsilon}\ln^2 n\right), \text{ since } n^{-\epsilon}\ln n\ln\ln n^{\epsilon} > 0 \\
	 & = O\left( n^{-\epsilon}\ln^2 n\right)\\
	 & = o(1).
	\end{align*}
	Thus, after combining two cases, we see that $nT_0(\alpha_r, \mu) \rightarrow 0$, and so the lemma follows.
\end{proof}

However, unlike the $nT_0(\alpha_r, \mu)$ term, which approaches 0, when $n \rightarrow \infty$, the term $(Ve^{\mu x}nx^{-\alpha_r}-\mu x) \rightarrow -\infty$, and the next lemma states this fact in a more rigorous way.

\begin{lemma}
	\label{lemma:W_2}
	Let $x=n^{\frac{1}{\alpha_r}+\epsilon}$ and $\mu=\frac{1}{x}\ln\frac{x^{\alpha_r}}{n}$ with $0<\alpha_r \leq 1$ and any $\epsilon>0$. Then
	$$
	Ve^{\mu x}nx^{-\alpha_r}-\mu x = V - \alpha_r\epsilon\ln n.
	$$	
\end{lemma}  
\begin{proof}
	Clearly,
	\begin{align*}
		Ve^{\mu x}nx^{-\alpha_r}-\mu x & = Ve^{\frac{x}{x}\ln \frac{x^{\alpha_r}}{n}}nx^{-\alpha_r}-\frac{x}{x}\ln\frac{x^{\alpha_r}}{n}\\
		& = V-\ln\frac{x^{\alpha_r}}{n}\\
		& = V-\ln\frac{n^{1+\alpha_r\epsilon}}{n}\\
		& = V - \alpha_r\epsilon\ln n. \qedhere
	\end{align*}
\end{proof}

Hence, after collecting results of Lemmas~\ref{lemma:W_1} and~\ref{lemma:W_2}, the inequality~\eqref{eq:p_u} transforms into
\begin{align*}
\Pr[S_{n} \geq x] & = nVx^{-\alpha_r} + \exp\Big( nT_0(\alpha_r, \mu) + \left(Ve^{\mu x}nx^{-\alpha_r}  - \mu x\right) \Big) \\
& \leq nVx^{-\alpha_r} + \exp\Big( o(1) + V - \alpha_r\epsilon\ln n \Big) \\
& \leq nVx^{-\alpha_r} + \exp\Big( 2V - \alpha_r\epsilon\ln n \Big) \\
& = nVx^{-\alpha_r} + e^{2V} n^{-\alpha_r\epsilon},
\end{align*}  
and after recalling that $x=n^{\frac{1}{\alpha_r}+\epsilon}$, we obtain the final form of the above inequality
\begin{align*}
\Pr[S_{n} \geq n^{\frac{1}{\alpha_r}+\epsilon}] & \leq \left(V+e^{2V}\right) n^{-\alpha_r\epsilon},
\end{align*}  
which proves Theorem~\ref{thr:bound_1}.
\end{proof}

Theorem~\ref{thr:bound_1} implies an obvious corollary that asserts an upper bound of values for $S_n$, when every $X_i$ in the aforementioned sum has 
a right-tail function that can be bounded by $V_i\,x^{-\alpha_{r,i}}$ with $0 < \alpha_{r,i} \leq 1$:

\begin{corollary}
Let $S_n=\sum_{i=1}^{n}X_i$, where $X_i \sim \mathbb{D}\left(\:\cdot\:, \alpha_{r} \right)$ are independent not necessary identically distributed integer-valued random variables, and $0< \alpha_r \leq 1$. Then w.h.p.
$$
S_n \leq C\,n^{1/\alpha_r},
$$
where $C>0$ is some constant.
\end{corollary}   

%And the corollary summarizes this section. 

Note, that in Theorem~\ref{thr:bound_1}, where we were interested in $S_n$ exceeding some \textit{positive} $x=n^{\frac{1}{\alpha_r}+\epsilon}$, we completely ignored the left tails of $X_i$'s. However, if we were trying to bound $S_n$ from below, then the left tails of the random variables $X_i$'s play a vital role, and the following theorem verifies this fact.

\begin{theorem}
	\label{thr:bound_2}
	Let $S_n=\sum_{i=1}^{n}X_i$, where $X_i \sim \mathbb{D}\left(\alpha_{l},\:\cdot\: \right)$ are independent not necessary identically distributed integer-valued random variables, and $0< \alpha_{l} \leq 1$. Then for any $\epsilon > 0$, we have
	\begin{align*}
	\Pr\left[S_n \leq -n^{\frac{1}{\alpha_l}+\epsilon} \right]\leq \left(W+e^{2W}\right) n^{-\alpha_l\epsilon},
	\end{align*}
	when $n \rightarrow \infty$.
\end{theorem}

\begin{proof}
	After establishing Theorem~\ref{thr:bound_1}, proof of the left tail bound is trivial. First, let's introduce random variables $X_i^{'}$ that have the same distributions as $-X_i$, i.e. $X_i^{'} \overset{d}{=} -X_i$. Clearly $X_i^{'} \sim \mathbb{D}\left( \:\cdot\:, \alpha_l\right)$, hence, from Definiton~\ref{def:D}, it follows that $F_{X_i^{'}+}(k)\leq W_{X_{i}}k^{-\alpha_{l}}$.
	
	Then for $x=n^{1/\alpha_r+\epsilon}$ we have
	\begin{align*}
		\Pr\left[S_n \leq -x\right] & = \Pr\left[ \sum_{i=1}^{n}X_i \leq -x\right] \\
		& = \Pr\left[ -\sum_{i=1}^{n}X_i \geq x\right]\\
		& = \Pr\left[ \sum_{i=1}^{n}-X_i \geq x\right]\\
		& = \Pr\left[ \sum_{i=1}^{n}X_i^{'} \geq x\right]\\
		& = \Pr\left[ S_n^{'} \geq x\right], \text{ where } S_n^{'}:=\sum_{i=1}^{n}X_i^{'}\\
		& \leq \left(W + e^{2W}\right)\,  n^{-\alpha_l\epsilon},
	\end{align*}
	where the last inequality follows after applying Theorem~\ref{thr:bound_1} to the sum $S_n^{'}$,  which consists of random variables with the right tail  functions that can be bounded by some $Wx^{-\alpha_l}$. What is left is to recall that $x=n^{1/\alpha_l+\epsilon}$, and the left-tail bound is established.
\end{proof}

Similarly to the right-tail bound of $S_n$, we can state a corollary that restricts the lower range of values of $S_n$.

\begin{corollary}
	Let $S_n=\sum_{i=1}^{n}X_i$, where $X_i \sim \mathbb{D}\left(\alpha_{l}, \:\cdot\: \right)$ are independent not necessary identically distributed integer-valued random variables, and $0< \alpha_l \leq 1$. Then w.h.p.
	$$
	-C\,n^{1/\alpha_l} \leq S_n,
	$$
	where $C>0$ is some constant.
\end{corollary}

\section{Tail bounds of $S_n-\E S_n$, when $\alpha_l, \alpha_r > 1$}

As was mentioned previously, Theorem~\ref{thr:bound_1} and Theorem~\ref{thr:bound_2} can be used to upper and lower bound the sum of integer-valued random variables, when the variables have left- or right-tail function that can be bounded in the ''best`` case by $Ck^{-\alpha}$ with constants $C>0$ and $0<\alpha\leq 1$. However, if the variables, which the sum consists of, have distributions from $\mathbb{D}(\alpha_l, \alpha_r)$ with both $\alpha_l > 1$ and $\alpha_r > 1$, then, these variables must have finite expectations, and therefore, the sum itself has finite expectation; moreover, as we show in Theorems~\ref{thr:bound_3} and~\ref{thr:bound_4}, the sum does not deviate much from its expected value. This result is summarized in the Corollary~\ref{cor:concentration}. 

Note, that results in this section require that both tails of every variable $X_i$ can be bounded by a power function with power $-\alpha$, where $\alpha>1$.  

First, let's verify that the variable $X \sim \mathbb{D}(>1,>1)$ has finite expectation. For that we provide an alternative way to compute expectation of random variables, which is a generalization of the tail sum of expectation:

\begin{lemma}[Generalized Tails Sum Formula]
\label{lemma:generalized-expectation}	
Let $X$ be a random variable with support on $\mathbb{Z}$, such that $\E X$ exists. Then
$$
\E X = \sum_{j=1}^{\infty}F_{X+}(j)-\sum_{j=1}^{\infty}F_{X-}(j).
$$	
\end{lemma}
\begin{proof}
	From the definition of expectation, it follows that
	\begin{align*}
	\E X & = \sum_{k}k\Pr\left[X=k\right]\\
	& = \sum_{k \leq -1}k\Pr\left[X=k\right]+\sum_{k\geq 1}k\Pr\left[X=k\right]\\
	& = \sum_{k=1}^{\infty}\left(-k\right)\Pr\left[X=-k\right]+\sum_{k=1}^{\infty}k\Pr\left[X=k\right]\\
	& = \sum_{k=1}^{\infty} k\Pr\left[X=k\right] - \sum_{k=1}^{\infty} k\Pr\left[X=-k\right]\\
	& = \sum_{k=1}^{\infty} \sum_{j=1}^{k}\Pr\left[X=k\right] - \sum_{k=1}^{\infty} \sum_{j=1}^{k}\Pr\left[X=-k\right]\\
	& = \sum_{j=1}^{\infty}\sum_{k=j}^{\infty}\Pr\left[X=k\right]-\sum_{j=1}^{\infty}\sum_{k=j}^{\infty}\Pr\left[X=-k\right]\\
	& = \sum_{j=1}^{\infty}\Pr\left[X \geq j\right]-\sum_{j=1}^{\infty}\Pr\left[X \leq -j\right].
	\end{align*}
	Recall that $F_{X+}(j):=\Pr\left[X_i \geq j\right]$ and $F_{X-}(j):=\Pr\left[X_i \leq -j\right]$ for every $j \geq 1$, and, therefore,
	\begin{align}
	\E X & = \sum_{j=1}^{\infty}\Pr\left[X \geq j\right]-\sum_{j=1}^{\infty}\Pr\left[X \leq -j\right] \nonumber \\
	& = \sum_{j=1}^{\infty}F_{X+}(j)-\sum_{j=1}^{\infty}F_{X-}(j)\label{eq:tail_sum}.
	\end{align}
	The lemma is proved.
\end{proof}

Now, after establishing an alternative way for calculating the expected value of a random variable, we can state a simple corollary, which verifies that random variables from $\mathbb{D}(>1,>1)$ have finite expectations.

\begin{corollary}
\label{lemma-expectation}
An integer-valued r.v. $X \sim \mathbb{D}(>1,>1)$ has finite expectation.
\end{corollary}
\begin{proof}
Since $X \sim \mathbb{D}(>1,>1)$, we have that
$$
F_{X+}(k) \leq V \, k^{-\alpha_r}, \quad \text{ and } \quad F_{X-}(k) \leq W\,k^{-\alpha_l} \text{ for any } k > 0, 
$$
where $V,W > 0$ and $\alpha_l,\alpha_r > 1$.

Next, introduce quantities 
$$
R=\sum_{k=1}^{\infty}F_{X+}(k) \quad \text{  and  } \quad L=\sum_{k=1}^{\infty}F_{X-}(k).
$$
Since $0\leq F_{X+}(k)\leq V\,k^{-\alpha_r}$ and $\alpha_r > 1$, we have that
$$
R = \sum_{k=1}^{\infty}F_{X+}(k) \leq V\sum_{k=1}^{\infty}k^{-\alpha_r} < \infty.
$$
Moreover, since $R$ is the sum of non-negative terms, we obtain that $0 \leq R < \infty$. The same idea we apply to $|L|$.

Since both $0 \leq R,|L| < \infty$ are finite, we obtain from~\ref{lemma:generalized-expectation} that
$$
|\E X| = \left|\sum_{j=1}^{\infty}F_{X+}(j)-\sum_{j=1}^{\infty}F_{X-}(j)\right| = |R - L| \leq |R| + |L| < \infty. 
$$  

Hence, when $X \sim \mathbb{D}(>1,>1)$, then $|\E X| < \infty$.

\end{proof}

Therefore, every random variable $X \sim \mathbb{D}\left(\alpha_l, \alpha_r\right)$ with $\alpha_l,\alpha_r>1$ has finite expectation.
We exploit this fact in order to obtain tighter bounds for the sum $S_n$. We show that whenever random variables have tail functions that can be bounded by $Vx^{-\alpha}$ with $\alpha>1$, then the sum of such variables is concentrated around its mean. 

The next two theorems assert this fact, where first we show that $S_n$ does not deviate much to the right from $\E S_n$, while the second theorem states a similar result but for the deviation to the left from the expectation.

\begin{theorem}
\label{thr:bound_3}
Let $S_n=\sum_{i=1}^{n}X_i$, where $X_i \sim \mathbb{D}\left(\alpha,\alpha \right)$ are independent not necessary identically distributed integer-valued random variables with $\alpha > 1$. Then for any $\epsilon > 0$, we have
$$
\Pr\left[S_n-\E S_n \geq n^{\max(1/\alpha,\,1/2) + \epsilon}\right]  \leq Vn^{1-\max(1,\,\alpha/2) - \alpha\epsilon} + e^{2V}n^{-\alpha\epsilon},
$$
when $n \rightarrow \infty$.
\end{theorem}
%\begin{discussion}
%\end{discussion}
\begin{proof}
The basic idea of the proof is similar to that of Theorem~\ref{thr:bound_1}, differing only in some details.
First, from Lemma~\ref{lemma-expectation}, it follows that $\left|\E X_i\right| < \infty$, and so
$$
\E S_n = \E \sum_{i=1}^{n}X_i=\sum_{i=1}^{n}\E X_i
$$ 
is finite as well.
Next, since each $X_i \sim \mathbb{D}\left(\alpha,\alpha\right)$ and $\alpha>1$, then, obviously, 
$$
F_{X_i+}(k)\leq V \, k^{-\alpha} \quad \text{ and }\quad F_{X_i-}(k)\leq W\, k^{-\alpha} 
$$
for any $k \geq 1$. 

Also we introduce a set of events $B_{i}:= \{X_{i} \leq x \}$ and 
$$
B:=\bigcap_{i=1}^{n}B_{i}.
$$
Then, likewise to the previous proof, we have
\begin{align*}
\Pr[S_{n} \geq \E S_n + x] & = \Pr[\{S_n \geq \E S_n + x\} \wedge \bar{B}] + \Pr[\{S_{n} \geq \E S_n + x\} \wedge B] \\
                           & \leq \Pr[\bar{B}] + \Pr[\{S_{n} \geq \E S_n + x\} \wedge B] \\
                           & \leq \Pr[\bar{B}] + \Pr\left[ S_{n} \geq \E S_n + x \: | \: B\right]\, \cdot \Pr\left[B\right].
\end{align*}

Since event $B$ is the intersection of independent events $\{ X_i \leq x \}$, we have 
$$
\Pr[B]=\prod_{i=1}^{n}\Pr[X_i \leq x];
$$ 
and to bound the probability of event $\bar{B}$ we apply Union bound, like we did in~\eqref{eq:pr_not_b}:
$$
\Pr[\bar{B}] \leq nVx^{-\alpha}.
$$ 

Hence, the above inequality of the probability $\Pr[S_n \geq \E S_n + x]$ can be further simplified
\begin{align*}
\Pr[S_{n} \geq \E S_n + x] & \leq \Pr[\bar{B}] + \Pr\left[ S_{n} \geq \E S_n + x \: | \: B\right]\, \cdot \Pr\left[B\right]\\
 & \leq nVx^{-\alpha} + \Pr\left[ S_{n} \geq \E S_n + x \: | \: B\right]\, \cdot \prod_{i=1}^{n}\Pr[X_i \leq x].
\end{align*}

Next, consider the sum $S_n$ in probability $\Pr\left[ S_{n} \geq \E S_n + x \: | \: B\right]$. This sum consists of $n$ independent random variables having conditionally the same distribution as the original variable $X_i$ but truncated at $x$. Formally speaking, let's introduce new ``nearly-centered''\footnote{The new variable is not fully centered, since $\E Z_i < \E X_i$} truncated random variables $Z_i$ with the following probability distribution function
\begin{align}
\Pr[Z_i = k - \E X_i] = \begin{cases}
\frac{\Pr[X_i = k]}{\Pr[X_i \leq x] }, & \text{when } k \leq x\\
0, & \text{otherwise.} 
\end{cases}\label{eq:z_i}
\end{align}
Then
$$
\Pr[S_{n} \geq \E S_n + x\:|\:B] = \Pr\left[\sum_{i=1}^{n}Z_i \geq x\right] =: \Pr\left[S_{n}^{ \langle x \rangle} \geq  x\right], \label{eq:truncated_sum}
$$  
where $Z_i$'s are the truncated versions of the respective r.v. $X_i$'s, and $S_{n}^{ \langle x \rangle}=\sum_{i=1}^{n}Z_i$. So now we have
\begin{align}
\Pr[S_{n} \geq \E S_n + x] & \leq nVx^{-\alpha} + \Pr\left[ S_{n} \geq \E S_n + x \: | \: B\right]\, \cdot \prod_{i=1}^{n}\Pr[X_i \leq x] \nonumber \\
& =  nVx^{-\alpha} + \Pr\left[S_{n}^{ \langle x \rangle} \geq  x\right]\, \cdot \prod_{i=1}^{n}\Pr[X_i \leq x].\label{eq:final_inequality_1}
\end{align}
As we did before, we apply generalized Markov's inequality to the probability of the sum of truncated variables to obtain the following inequality
\begin{align*}
\Pr\left[S_{n}^{ \langle x \rangle} \geq  x\right] & \leq \frac{\E e^{\mu S_{n}^{ \langle x \rangle}}}{e^{\mu x}}, 
\end{align*}
which holds for any  $\mu \geq 0$, however, in what follows, we will require that $\mu \rightarrow 0$ when $n \rightarrow \infty$ (moreover, we will assume that $n$ is large enough for $\mu < 1$). Hence,
\begin{align*}
\Pr\left[S_{n}^{ \langle x \rangle} \geq  x\right] & \leq \frac{\E e^{\mu S_{n}^{ \langle x \rangle}}}{e^{\mu x}}\\
& = e^{-\mu x}\E e^{\mu \sum_{i=1}^{n}Z_i  }\\
& = e^{-\mu x}\prod_{i=1}^{n}\E e^{\mu Z_i}, \text{ since all $Z_i$'s are independent}\\
& = e^{-\mu x}\prod_{i=1}^{n}\sum_{k \leq x}e^{\mu (k-\E X_i)}\Pr[Z_i=k-\E X_i].
\end{align*}
Next, we use the definition of the variable $Z_i$ from~\eqref{eq:z_i}
\begin{align*}
\Pr\left[S_{n}^{ \langle x \rangle} \geq  x\right] & \leq e^{-\mu x}\prod_{i=1}^{n}\sum_{k \leq x}e^{\mu (k-\E X_i)}\Pr[Z_i=k-\E X_i]\\
& \leq e^{-\mu x}\prod_{i=1}^{n}\sum_{k \leq x}e^{\mu (k-\E X_i)}\frac{\Pr\left[X_i=k\right]}{\Pr\left[X_i \leq x\right]}\\
& = e^{-\mu x}\frac{\prod_{i=1}^{n}\sum_{k \leq x}e^{\mu (k-\E X_i)}\Pr\left[X_i=k\right]}{\prod_{i=1}^{n}\Pr[X_i \leq x]}   
\end{align*}

Thus, we can substitute $\Pr[S_{n}^{ \langle x \rangle} \geq  x]$ in~\eqref{eq:final_inequality_1} with the above calculated inequality to obtain
\begin{align}
\Pr[S_{n} \geq \E S_n + x] & \leq nVx^{-\alpha} + \Pr\left[S_{n}^{ \langle x \rangle} \geq  x\right]\, \cdot \prod_{i=1}^{n}\Pr[X_i \leq x]\nonumber\\
& \leq nVx^{-\alpha} + e^{-\mu x}\frac{\prod_{i=1}^{n}\sum_{k\leq x}e^{\mu (k-\E X_i)}\Pr\left[X_i=k\right]}{\prod_{i=1}^{n}\Pr[X_i \leq x]}\, \cdot \prod_{i=1}^{n}\Pr[X_i \leq x]\nonumber\\
& = nVx^{-\alpha} + e^{-\mu x}\prod_{i=1}^{n}\sum_{k \leq x}e^{\mu (k-\E X_i)}\Pr\left[X_i=k\right]\nonumber\\
& = nVx^{-\alpha} + e^{-\mu x}\prod_{i=1}^{n}e^{-\mu\E X_i}\sum_{k \leq x}e^{\mu k}\Pr\left[X_i=k\right]\nonumber\\
& =: nVx^{-\alpha} + e^{-\mu x}\prod_{i=1}^{n}e^{-\mu\E X_i}P_i(\mu,x). \label{eq:almost_final}
\end{align}
where 
\begin{align*}
P_i(\mu, x)& :=\sum_{k \leq x}e^{\mu k}\Pr\left[X_i=k\right].
\end{align*}
To obtain a bound on $P_i(\mu,x)$, we split its summation into 3 disjoint intervals, and bound each interval separately
\begin{align}
P_i(\mu, x) 	& = \sum_{k \leq x}e^{\mu k}\Pr\left[X_i=k\right] \nonumber\\
			& \leq \sum_{k \leq 0}e^{\mu k}\Pr\left[X_i=k\right]  + \sum_{k=1}^{\lfloor M \rfloor}e^{\mu k}\Pr\left[X_i=k\right]
							+ \sum_{k=\lceil M \rceil}^{x}e^{\mu k}\Pr\left[X_i=k\right] \nonumber \\
			& =:  J_{i,1} + J_{i,2} + J_{i,3}, \label{eq:R}
\end{align}
where $M=\frac{2\alpha}{\mu}$ (again, since we would like to have $0 \ll M < x$, we require $\mu x \rightarrow \infty$), and the next three sections identify bounds for each interval, and the fourth section assembles upper bounds of $J_{i,1}, J_{i,2}$, and $J_{i,3}$ to provide a bound for the probability $\Pr[S_n - \E S_n \geq x]$. 

\subsection{Upper bound of $J_{i,1}$, when $\alpha>1$}

Observe, that $P_i(\mu, x)$ is multiplied by $e^{-\mu\E X_i}$ term in~\eqref{eq:almost_final}. When $\E X_i \geq 0$, this shouldn't cause any troubles, however, if $\E X_i<0$, then clearly $e^{-\mu\E X_i}=1+\epsilon'$ with $\epsilon'>0$. Taking into account that $e^{-\mu\E X_i}P_i(\mu, x)$ stands under product operator, this extra $(1+\epsilon')$ term, when raised to the power of $n\rightarrow\infty$ may cause troubles if not dealt with properly. 

So our goal, beside obtaining bounds of $P_i(\mu, x)$ in terms of tail functions, is to ''extract`` and include into the bound a term that will eventually diminish the $e^{-\mu\E X_i}$ multiplier. 

The next lemma demonstrates how we achieve this goal by bounding the $J_{i,1}$ interval with the tail functions \textit{and} partial expectation of $X_i$ (the other part of $\E X_i$ will be included in the second interval). 

\begin{lemma}
\label{lemma:ji1}
  Let $X_i \sim \mathbb{D}(\alpha,\alpha)$ be an integer valued r.v. with $\alpha>1$. Then the term
  $$
  J_{i,1} = \sum_{k \leq 0}e^{\mu k}\Pr[X_i=k],
  $$ 
  where $0 < \mu < 1$ is bounded from above by
$$
J_{i,1} \leq \Pr[X_i \leq 0]+\mu\sum_{k \leq 0}k\Pr[X_i=k] + \begin{cases}
O\left(\mu^{\alpha}\right) , &\text{ when } 1 < \alpha < 2,\\
O\left(-\mu^2\ln \mu\right), & \text{ when } \alpha=2,\\
O\left(\mu^2\right),& \text{ when } \alpha > 2.
\end{cases}
$$
\end{lemma}
\begin{proof}
We have

\begin{align}
J_{i,1} & = \sum_{k \leq 0}e^{\mu k}\Pr[X_i=k] \nonumber \\
		& = \sum_{k \leq 0}\Big((1+\mu k) + e^{\mu k}-(1+\mu k)\Big)\Pr[X_i=k] \nonumber \\
		%& = \sum_{k \leq 0}\Pr[X_i=k]+\mu\sum_{k \leq 0}k\Pr[X_i=k] + \sum_{k \leq 0}\Big(e^{\mu k}-1-\mu k\Big)\Pr[X_i=k] \nonumber \\
		& = \Pr[X_i \leq 0]+\mu\sum_{k \leq 0}k\Pr[X_i=k] + \sum_{k \leq 0}\Big(e^{\mu k}-1-\mu k\Big)\Pr[X_i=k] \label{eq:j_1}. 
\end{align}

Consider the rightmost sum. First, denote by $\phi(k)$ the function 
\begin{align}
\phi(k) := e^{\mu k}-1-\mu k, \label{eq:phi}
\end{align}
and then apply summation by parts (Lemma~\ref{lemma:summation-by-parts})  
\begin{align*}
\sum_{k \leq 0}\Big(e^{\mu k}-1-\mu k\Big)\Pr[X_i=k] & = \phi(0)\Pr[X_i \leq 0]+ \sum_{k \leq -1}\Big(\phi(k)-\phi(k+1)\Big)\Pr[X_i \leq k]\\
& =  \sum_{k \geq 1}\Big(\phi(-k)-\phi(1-k)\Big)\Pr[X_i \leq -k], \text{ since } \phi(0)=0.
\end{align*}
Next, for every $X_i$ and $k \geq 1$ we have 
$$
\Pr\left[X_i \leq -k\right]=F_{X_i-}(k) \leq Wk^{-\alpha}, 
$$
and, therefore,
\begin{align*}
\sum_{k \leq 0}\Big(e^{\mu k}-1-\mu k\Big)\Pr[X_i=k] & = \sum_{k \geq 1}\Big(\phi(-k)-\phi(1-k)\Big)\Pr[X_i \leq -k]\\
& \leq W\sum_{k \geq 1}\Big(\phi(-k)-\phi(1-k)\Big)k^{-\alpha}.
\end{align*}

After substituting back the function we denoted by $\phi$~\eqref{eq:phi}, we obtain that
\begin{align*}
\phi(-k)-\phi(1-k) & = e^{-\mu k}-1+\mu k - e^{\mu (1-k)}+1+\mu(1-k)\\
%& = e^{-\mu k} - e^{\mu (1-k)}+\mu\\
& = e^{-\mu k}\left(1 - e^{\mu}\right)+\mu\\
& \leq -\mu e^{-\mu k}+\mu, \quad \text{ since } e^{\mu} \geq 1 + \mu\\
& = \mu\left(1 - e^{-\mu k}\right).
\end{align*}

Thus, 
\begin{align*}
\sum_{k \leq 0}\Big(e^{\mu k}-1-\mu k\Big)\Pr[X_i=k] & \leq W\sum_{k\geq 1}\Big(\phi(-k)-\phi(1-k)\Big)k^{-\alpha} \nonumber \\
& \leq W\mu\sum_{k \geq 1}\left(1-e^{-\mu k}\right)k^{-\alpha} \nonumber \\
& \leq W\mu\left(  \sum_{k = 1}^{\lfloor 1/\mu \rfloor}(1-e^{-\mu k})k^{-\alpha} + \sum_{k = \lceil 1/\mu \rceil}^{\infty}(1-e^{-\mu k})k^{-\alpha} \right). \nonumber \\
& \leq W\mu\left(  \sum_{k = 1}^{\lfloor 1/\mu \rfloor}(1-e^{-\mu k})k^{-\alpha} + \sum_{k = \lceil 1/\mu \rceil}^{\infty}k^{-\alpha}\right) \nonumber \\
& \leq W\mu\left(  \mu \sum_{k = 1}^{\lfloor 1/\mu \rfloor}k^{1-\alpha} + \sum_{k = \lceil 1/\mu \rceil}^{\infty}k^{-\alpha}\right), \text{ since } e^{-\mu k} \geq 1 - \mu k. \nonumber
\end{align*}

Now we have two sums of strictly decreasing functions, which we bound using the integral upper bound of summation:

\begin{align*}
\sum_{k \leq 0}\Big(e^{\mu k}-1-\mu k\Big)\Pr[X_i=k] & \leq W\mu\left(  \mu \sum_{k = 1}^{\lfloor 1/\mu \rfloor}k^{1-\alpha} + \sum_{k = \lceil 1/\mu \rceil}^{\infty}k^{-\alpha}\right) \nonumber \\
& \leq W\mu\left(  \mu + \mu \sum_{k = 2}^{\lfloor 1/\mu \rfloor}k^{1-\alpha} + \left(\lceil 1/\mu \rceil\right)^{-\alpha} + \sum_{k = \lceil 1/\mu \rceil+1}^{\infty}k^{-\alpha}\right) \nonumber \\
%& \leq W\mu\left(  \mu + \mu\int\limits_{1}^{\lfloor 1/\mu \rfloor}x^{1-\alpha} \diff x + \mu^{\alpha} + \int\limits_{\lceil 1/\mu \rceil}^{\infty}x^{-\alpha} \diff x \right) \nonumber \\
& \leq W\mu\left(  \mu + \mu\int\limits_{1}^{1/\mu}x^{1-\alpha} \diff x + \mu^{\alpha} + \int\limits_{ 1/\mu }^{\infty}x^{-\alpha} \diff x \right) \nonumber \\
%& \leq W\mu\left(  \mu  + \mu^{\alpha} + \frac{1}{\alpha-1}\mu^{\alpha-1} +\mu\int\limits_{1}^{1/\mu}x^{1-\alpha} \diff x \right) \nonumber \\
& \leq W\left(  \mu^2  + \mu^{\alpha+1} + \frac{1}{\alpha-1}\mu^{\alpha} +\mu^2\int\limits_{1}^{1/\mu}x^{1-\alpha} \diff x \right) \nonumber \\
& =: W\left(  \mu^2  + \mu^{\alpha+1} + \frac{1}{\alpha-1}\mu^{\alpha} +\sigma(\alpha, \, \mu) \right), \nonumber
\end{align*}
where we define $\sigma(\alpha, \, \mu) = \mu^2\int\limits_{1}^{1/\mu}x^{1-\alpha} \diff x$, which is upper bounded by
\begin{align*}
\sigma(\alpha, \, \mu) & = \mu^2\int\limits_{1}^{1/\mu}x^{1-\alpha} \diff x\\
& = \mu^2\begin{cases}
\frac{\mu^{\alpha-2}}{2-\alpha}-\frac{1}{2-\alpha}, &\text{ when } 1 < \alpha < 2,\\
-\ln \mu, & \text{ when } \alpha=2,\\
\frac{1}{\alpha-2}-\frac{\mu^{\alpha-2}}{\alpha-2},& \text{ when } \alpha > 2
\end{cases} \\
& = \begin{cases}
O\left(\mu^{\alpha}\right), &\text{ when } 1 < \alpha < 2,\\
O\left(-\mu^2\ln \mu\right), & \text{ when } \alpha=2,\\
O(\mu^2),& \text{ when } \alpha > 2.
\end{cases}
\end{align*}

Thus, we have that
\begin{align*}
\sum_{k \leq 0}\Big(e^{\mu k}-1-\mu k\Big)\Pr[X_i=k] & \leq W\left(  \mu^2  + \mu^{\alpha+1} + \frac{1}{\alpha-1}\mu^{\alpha} +\sigma(\alpha, \, \mu) \right) \nonumber \\
& \leq W\begin{cases}
\mu^2  + \mu^{\alpha+1} + \frac{1}{\alpha-1}\mu^{\alpha} +O\left(\mu^{\alpha}\right), &\text{ when } 1 < \alpha < 2,\\
2\mu^2  + \mu^{3} +O\left(-\mu^2\ln \mu\right), & \text{ when } \alpha=2,\\
\mu^2  + \mu^{\alpha+1} + \frac{1}{\alpha-1}\mu^{\alpha} + O\left(\mu^2\right),& \text{ when } \alpha > 2.
\end{cases} \nonumber \\
& = \begin{cases}
O\left(\mu^{\alpha}\right), &\text{ when } 1 < \alpha < 2,\\
O\left(-\mu^2\ln \mu\right), & \text{ when } \alpha=2,\\
O(\mu^2),& \text{ when } \alpha > 2,
\end{cases}
\end{align*}
since $\mu < 1$; therefore, the first interval $J_{i,1}$~\eqref{eq:j_1} is upper bounded by
\begin{align*}
J_{i,1} & \leq \Pr[X_i \leq 0]+\mu\sum_{k \leq 0}k\Pr[X_i=k] + \sum_{k \leq 0}\Big(e^{\mu k}-1-\mu k\Big)\Pr[X_i=k]\\
& \leq \Pr[X_i \leq 0]+\mu\sum_{k \leq 0}k\Pr[X_i=k] + \begin{cases}
O\left(\mu^{\alpha}\right), &\text{ when } 1 < \alpha < 2,\\
O\left(-\mu^2\ln \mu\right), & \text{ when } \alpha=2,\\
O(\mu^2),& \text{ when } \alpha > 2,
\end{cases} 
\end{align*}
and the lemma follows.
\end{proof}

\subsection{Upper bound of $J_{i,2}$, when $\alpha>1$}

The basic idea of the proof in this section resembles the one presented in Lemma~\ref{lemma:ji1}. First and foremost, we add and subtract the $(1+\mu k)$ term to obtain the positive part of the $\E X_i$, and after that we apply summation by parts followed with simple integral bound.

\begin{lemma}
\label{lemma:ji2}
 Let $X_i \sim \mathbb{D}(\alpha,\alpha)$ be an integer valued r.v. with $\alpha > 1$. When $M=\frac{2\alpha}{\mu}$ and $0<\mu<1$, then the term
 $$
 J_{i,2} = \sum_{k=1}^{\lfloor M \rfloor}e^{\mu k}\Pr[X_i=k], 
 $$ 
is bounded from above by
\begin{align*}
J_{i,2} \leq \Pr[X_i \geq 1] + \mu\sum_{k=1}^{\infty}k\Pr[X_i=k] +   \begin{cases}
O\left(\mu^{\alpha}\right), & \text{when } 1 < \alpha < 2, \\
O\left(- \mu^2\ln \mu\right), & \text{when } \alpha = 2,\\
O\left(\mu^{2}\right), & \text{when } \alpha > 2.
\end{cases}
\end{align*}
\end{lemma}
\begin{proof}
For $J_{i,2}$ we have 
\begin{align}
J_{i,2} 	& = \sum_{k=1}^{\lfloor M \rfloor}e^{\mu k}\Pr[X_i=k] \nonumber \\
		& = \sum_{k=1}^{\lfloor M \rfloor}\Big(1 + \mu k+ e^{\mu k} - 1 - \mu k\Big)\Pr[X_i=k] \nonumber  \\
		& = \sum_{k=1}^{\lfloor M \rfloor}\Pr[X_i=k] + \mu\sum_{k=1}^{\lfloor M \rfloor}k\Pr[X_i=k] + \sum_{k=1}^{\lfloor M \rfloor}\Big(e^{\mu k} - 1 - \mu k\Big)\Pr[X_i=k] \nonumber  \\
		%& \leq \sum_{k=1}^{\infty}\Pr[X_i=k] + \mu\sum_{k=1}^{\infty}k\Pr[X_i=k] + \sum_{k=1}^{\lfloor M \rfloor}\Big(e^{\mu k} - 1 - \mu k\Big)\Pr[X_i=k] \nonumber \\
		& \leq \Pr[X_i \geq 1] + \mu\sum_{k=1}^{\infty}k\Pr[X_i=k] + \sum_{k=1}^{\lfloor M \rfloor}\Big(e^{\mu k} - 1 - \mu k\Big)\Pr[X_i=k] \nonumber \\
		& =: \Pr[X_i \geq 1] + \mu\sum_{k=1}^{\infty}k\Pr[X_i=k] + J_{i,2}^{0} \label{eq:j_i_2},
\end{align} 
where $J_{i,2}^{0}=\sum_{k=1}^{\lfloor M \rfloor}\Big(e^{\mu k} - 1 - \mu k\Big)\Pr[X_i=k]$. First we denote 
 $$
 \phi(k):=e^{\mu k}-1-\mu k,
 $$
 and after applying summation by parts (Lemma~\ref{lemma:summation-by-parts}), we obtain
\begin{align*}
J_{i,2}^{0} &= \sum_{k=1}^{\lfloor M \rfloor}\Big(e^{\mu k} - 1 - \mu k\Big)\Pr[X_i=k] \\
& = \sum_{k=1}^{\lfloor M \rfloor}\phi(k)\Pr[X_i=k] \\
& = \sum_{k=0}^{\lfloor M \rfloor}\phi(k)\Pr[X_i=k], \quad  \text{ since } \phi(0)=0\\
& = \phi(0)\Pr\left[X_i \geq 0\right] - \phi(\lfloor M \rfloor)\Pr\left[X_i \geq \lfloor M \rfloor\right] + \sum_{k=2}^{\lfloor M \rfloor}\Big(\phi(k) - \phi(k-1)  )\Pr[X_i \geq k]\\
& \leq \sum_{k=1}^{\lfloor M \rfloor}\Big(\phi(k) - \phi(k-1)  )\Pr[X_i \geq k].
\end{align*}			
Now, since $\phi(k)=e^{\mu k} - 1 - \mu k$, we further have
\begin{align*}
J_{i,2}^{0}	& \leq \sum_{k=1}^{\lfloor M \rfloor}\Big(\phi(k)-\phi(k-1)\Big)\Pr\left[X_i \geq k\right] \\
%& = \sum_{k=1}^{\lfloor M \rfloor}\Big(e^{\mu k} - 1 - \mu k-e^{\mu(k-1)}+1+\mu(k-1)\Big)\Pr\left[X_i \geq k\right]\\
%& = \sum_{k=1}^{\lfloor M \rfloor}\Big(e^{\mu k} - e^{\mu(k-1)}-\mu\Big)\Pr\left[X_i \geq k\right]\\
& = \sum_{k=1}^{\lfloor M \rfloor}\Big(e^{\mu k}(1-e^{-\mu})-\mu\Big)\Pr\left[X_i \geq k\right]\\
& \leq \sum_{k=1}^{\lfloor M \rfloor}\Big(\mu e^{\mu k}-\mu\Big)\Pr\left[X_i \geq k\right], \text{ since } e^{-\mu} \geq 1-\mu\\
& \leq \mu\sum_{k=1}^{\lfloor M \rfloor}\Big( e^{\mu k}-1\Big)\Pr\left[X_i \geq k\right].
\end{align*} 
Next, from the definition of the right-tail function, we have
$$
\Pr[X_i \geq k]=: F_{X_i+}(k) \leq Vk^{-\alpha} \text{ for every } k \geq 1.
$$
Hence,
\begin{align*}
J_{i,2}^{0} & \leq \mu\sum_{k=1}^{\lfloor M \rfloor}\Big( e^{\mu k}-1\Big)\Pr\left[X_i \geq k\right]\\
& \leq V\mu\sum_{k=1}^{\lfloor M \rfloor}\Big( e^{\mu k}-1\Big)k^{-\alpha}\\
& = V\mu\Big( e^{\mu}-1\Big) + V\mu\sum_{k=1}^{\lfloor M \rfloor}\Big( e^{\mu k}-1\Big)k^{-\alpha}.
\end{align*}

Now, since $0<\mu < 1$, we have that $e^{\mu} \leq 1 +2\mu$, and so
\begin{align*}
J_{i,2}^{0} & \leq V\mu\Big( e^{\mu}-1\Big) + V\mu\sum_{k=1}^{\lfloor M \rfloor}\Big( e^{\mu k}-1\Big)k^{-\alpha}\\
& \leq 2V\mu^2 + V\mu\sum_{k=1}^{\lfloor M \rfloor}\Big( e^{\mu k}-1\Big)k^{-\alpha}.
\end{align*}

Furthermore, observe that the function $f(k)=e^{\mu k} - 1$ is convex, hence, we can bound it from above with a straight line $l(k)=\frac{k}{M}(e^{2\alpha}-1)$ when $k \in [0...M]$. Therefore,

\begin{align*}
J_{i,2}^{0} & \leq  2V\mu^2 + V\mu\sum_{k=1}^{\lfloor M \rfloor}\Big( e^{\mu k}-1\Big)k^{-\alpha}\\
& \leq 2V\mu^2 + V\mu\frac{e^{2\alpha}-1}{M}\sum_{k=1}^{\lfloor M \rfloor}k^{1-\alpha}\\
& \leq 2V\mu^2 + V\mu^2\frac{e^{2\alpha}-1}{2\alpha}\sum_{k=1}^{\lfloor M \rfloor}k^{1-\alpha}, \quad \text{ since } M=\frac{2\alpha}{\mu}\\
& \leq 2V\mu^2 + Ve^{2\alpha}\mu^2\sum_{k=1}^{\lfloor M \rfloor}k^{1-\alpha}.
\end{align*}

Next, we apply the integral upper bound of summation of a decreasing function~\eqref{eq:integral-bound}. So for $\alpha > 1$ we have
\begin{align*}
J_{i,2}^{0} & \leq 2V\mu^2 + Ve^{2\alpha}\mu^2\sum_{k=1}^{\lfloor M \rfloor}k^{1-\alpha}\\
& = 2V\mu^2 + Ve^{2\alpha}\mu^2 + Ve^{2\alpha}\mu^2\sum_{k=2}^{\lfloor M \rfloor}k^{1-\alpha}\\
& \leq  O\left(\mu^2\right) + Ve^{2\alpha}\mu^2 \int\limits_{1}^{ M }x^{1-\alpha} \diff x\\
& =  O\left(\mu^2\right) +  Ve^{2\alpha}\mu^2 \begin{cases}
\frac{M^{2-\alpha}}{2-\alpha}-\frac{1}{2-\alpha}, & \text{ when } 1 < \alpha < 2,\\
\ln M, & \text{ when } \alpha=2,\\
\frac{1}{\alpha-2}-\frac{M^{2-\alpha}}{\alpha-2}, &\text{ when } \alpha>2.
\end{cases} %\\
%& \leq  2V\mu^2 + Ve^{1+e/2} \mu^{\alpha+1} + Ve^{2\alpha}\mu^2 \begin{cases}
%\frac{M^{2-\alpha}}{2-\alpha}, & \text{ when } 1 < \alpha < 2,\\
%\ln M, & \text{ when } \alpha=2,\\
%\frac{1}{\alpha-2}, &\text{ when } \alpha>2.
%\end{cases}
\end{align*}
Recall that $M=2\alpha/\mu$. Then
\begin{align*}
J_{i,2}^{0} & \leq  O\left(\mu^2\right) +  Ve^{2\alpha}\mu^2 \begin{cases}
\frac{M^{2-\alpha}}{2-\alpha}-\frac{1}{2-\alpha}, & \text{ when } 1 < \alpha < 2,\\
\ln M, & \text{ when } \alpha=2,\\
\frac{1}{\alpha-2}-\frac{M^{2-\alpha}}{\alpha-2}, &\text{ when } \alpha>2.
\end{cases}\\
& = \begin{cases}
O\left(\mu^{\alpha}\right), & \text{ when } 1 < \alpha < 2,\\
O\left(-\mu^2\ln \mu\right), & \text{ when } \alpha=2,\\
O\left(\mu^2\right), &\text{ when } \alpha>2,
\end{cases}
\end{align*}
where the last relation follows from the fact that $0 < \mu < 1$.

Thus, from~\eqref{eq:j_i_2}, it follows that the second interval is upper bounded by
\begin{align*}
J_{i,2} & \leq \Pr[X_i \geq 1] + \mu\sum_{k=1}^{\infty}k\Pr[X_i=k] + J_{i,2}^{0}\\
& = \Pr[X_i \geq 1] + \mu\sum_{k=1}^{\infty}k\Pr[X_i=k] + \begin{cases}
O\left(\mu^{\alpha}\right), & \text{ when } 1 < \alpha < 2,\\
O\left(-\mu^2\ln \mu\right), & \text{ when } \alpha=2,\\
O\left(\mu^2\right), &\text{ when } \alpha>2,
\end{cases}
\end{align*}
which proves the lemma.
\end{proof}

\subsection{Upper bound of $J_{i,3}$, when $\alpha > 1$}

This interval is the easiest to deal with. Note that $J_{i,3}$ is identical to the $I_{i,2}$ term from Lemma~\ref{lemma:Ii2}, and thus, we can re-use its result to bound the interval.

\begin{lemma}
\label{lemma:ji3}
  Let $X_i \sim \mathbb{D}(\alpha,\alpha)$ be an integer valued r.v. with $\alpha > 1$. When  $M=\frac{2\alpha}{\mu}$ and $0<\mu<1$, then the term
 $$
 J_{i,3} = \sum_{k=\lceil M \rceil}^{x}e^{\mu k}\Pr[X_i=k], 
 $$ 
 is bounded from above by
$$
J_{i,3} \leq O\left(\mu^{\alpha}\right) +  Ve^{\mu x}x^{-\alpha}. 
$$
\end{lemma}
\begin{proof}
 Proof follows from Lemma~\ref{lemma:Ii3}. 
\end{proof}

\subsection{Final assembling steps, when $\alpha > 1$}
Recall, that our goal was to bound $P_i(\mu,x) \leq J_{i,1}+J_{i,2}+J_{i,3}$ in~\eqref{eq:almost_final}. Thus, after collecting results of Lemmas~\ref{lemma:ji1},~\ref{lemma:ji2}, and~\ref{lemma:ji3}:
\begin{align*}
J_{i,1} & \leq \Pr[X_i \leq 0]+\mu\sum_{k \leq 0}k\Pr[X_i=k] + \begin{cases}
O\left(\mu^{\alpha}\right) , &\text{ when } 1 < \alpha < 2,\\
O\left(-\mu^2\ln \mu\right), & \text{ when } \alpha=2,\\
O\left(\mu^2\right),& \text{ when } \alpha > 2,
\end{cases},\\
J_{i,2} & \leq \Pr[X_i \geq 1] + \mu\sum_{k=1}^{\infty}k\Pr[X_i=k] +   \begin{cases}
O\left(\mu^{\alpha}\right), & \text{when } 1 < \alpha < 2, \\
O\left(- \mu^2\ln \mu\right), & \text{when } \alpha = 2,\\
O\left(\mu^{2}\right), & \text{when } \alpha > 2
\end{cases},\\
J_{i,3} & \leq O\left(\mu^{\alpha}\right) +  Ve^{\mu x}x^{-\alpha},
\end{align*}
we  obtain
\begin{align*}
P_i(\mu,x) 	& \leq J_{i,1} + J_{i,2} + J_{i,3} \\
			& = \sum_{k}\Pr[X_i=k] + \mu \sum_{k}k\Pr[X_i=k] +Ve^{\mu x}x^{-\alpha}
			+\begin{cases}
			O\left(\mu^{\alpha}\right), & \text{when } 1 < \alpha < 2, \\
			O\left(- \mu^2\ln \mu\right), & \text{when } \alpha = 2,\\
			O\left(\mu^{2}\right), & \text{when } \alpha > 2
			\end{cases}\\
			& = 1 + \mu \E X_i +Ve^{\mu x}x^{-\alpha}
						+\begin{cases}
						O\left(\mu^{\alpha}\right), & \text{when } 1 < \alpha < 2, \\
						O\left(- \mu^2\ln \mu\right), & \text{when } \alpha = 2,\\
						O\left(\mu^{2}\right), & \text{when } \alpha > 2
						\end{cases}\\
			& =: 1 + \mu\E X_i + Ve^{\mu x}x^{-\alpha} + \mathcal{T}(\alpha, \, \mu),			
\end{align*}
where 
$$
\mathcal{T}(\alpha, \, \mu) = \begin{cases}
						O\left(\mu^{\alpha}\right), & \text{when } 1 < \alpha < 2, \\
						O\left(- \mu^2\ln \mu\right), & \text{when } \alpha = 2,\\
						O\left(\mu^{2}\right), & \text{when } \alpha > 2.
						\end{cases}
$$

Hence,  after applying the well-known relation $1+x \leq e^{x}$, we obtain
\begin{align}
P_i(\mu,x) & \leq 1+ \mu \E X_i +Ve^{\mu x}x^{-\alpha}+\mathcal{T}(\alpha,\,\mu)
 \leq \exp\Big(\mu \E X_i +Ve^{\mu x}x^{-\alpha}+\mathcal{T}(\alpha,\,\mu)\Big) \label{eq:bound_P}.
\end{align}
Now, let's recall inequality~\eqref{eq:almost_final} and denote its right-hand side by $\mathcal{K}(\mu, x)$, that is
\begin{align}
\Pr[S_{n} \geq \E S_n + x] & \leq nVx^{-\alpha_r} + e^{-\mu x}\prod_{i=1}^{n}e^{-\mu\E X_i}P_i(\mu,x) \nonumber \\
& =: nVx^{-\alpha} + \mathcal{K}(\mu,x), \label{eq:K} 
\end{align}
where $\mathcal{K}(\mu,x):=e^{-\mu x}\prod_{i=1}^{n}e^{-\mu\E X_i}P_i(\mu,x)$.

Next, after having obtained bound for  $P_i(\mu,x)$~\eqref{eq:bound_P}, we can simplify $\mathcal{K}(\mu,x)$
\begin{align}
\mathcal{K}(\mu,x) & =e^{-\mu x}\prod_{i=1}^{n}e^{-\mu\E X_i}P_i(\mu,x) \nonumber \\
& \leq e^{-\mu x}\prod_{i=1}^{n}e^{-\mu\E X_i}\exp\Big(\mu \E X_i +Ve^{\mu x}x^{-\alpha}+\mathcal{T}(\alpha,\,\mu)\Big) \nonumber \\
& = e^{-\mu x}\prod_{i=1}^{n}\exp\Big(Ve^{\mu x}x^{-\alpha}+\mathcal{T}(\alpha,\,\mu)\Big) \nonumber \\
& = e^{-\mu x}\exp\Big(Vne^{\mu x}x^{-\alpha}+n\mathcal{T}(\alpha,\,\mu)\Big) \nonumber \\
& = \exp\Big(-\mu x + Vne^{\mu x}x^{-\alpha}+n\mathcal{T}(\alpha,\,\mu)\Big). \label{eq:K-tmp}
\end{align}

Next, we need to fix the value of $\mu$ that minimizes the above exponent, while keeping $\mu \rightarrow 0$, but $\mu x \rightarrow \infty$. Similar to the case when $0 < \alpha_r \leq 1$, we take $\mu=\frac{1}{x}\ln\frac{x^{\alpha}}{n}$. 

The next lemma verifies that the chosen value of $\mu$ satisfies both constraints.

\begin{lemma}
Let $\mu = \frac{1}{x}\ln\frac{x^{\alpha}}{n}$ and $x=n^{\max(1/\alpha,\:1/2)+\epsilon}$ for $\epsilon>0$. Then $\mu \rightarrow 0$ and  $\mu x \rightarrow \infty$ when $n \rightarrow \infty$.	
\end{lemma}

\begin{proof}
First, let's check that $\mu \rightarrow 0$. We have for $\mu$
\begin{align*}
\mu & = \frac{1}{x}\ln\frac{x^{\alpha}}{n}\\
& =\frac{1}{n^{\max(1/\alpha,\:1/2)+\epsilon}}\ln\frac{n^{\alpha\max(1/\alpha,\:1/2)+\alpha\epsilon}}{n}\\
& \leq \frac{1}{n^{1/2}}\ln\frac{n^{\alpha/\min(\alpha,\:2)+\alpha\epsilon}}{n}.
\end{align*}	
However, $\frac{\alpha}{\min(\alpha,\:2)} \geq 1$, for any $\alpha>0$, and so $\frac{\alpha}{\min(\alpha,\:2)} =1+\delta$, where $\delta \geq 0$. Then
\begin{align*}
\mu & \leq \frac{1}{n^{1/2}}\ln\frac{n^{\alpha/\min(\alpha,\:2)+\alpha\epsilon}}{n}\\
& \leq \frac{1}{n^{1/2}}\ln\frac{n^{1+\delta+\alpha\epsilon}}{n}\\
& \leq \frac{1}{n^{1/2}}O(\ln n)\\
& = o(1),
\end{align*}
when $n \rightarrow \infty$. 

Now, in a similar way, we prove that $\mu x \rightarrow \infty$. From the defined values of $\mu$ and $x$, it follows that
\begin{align*}
\mu x & = \frac{x}{x}\ln\frac{x^{\alpha}}{n}\\
& = \ln\frac{x^{\alpha}}{n}\\
& = \ln\frac{n^{\alpha\max(1/\alpha,\:1/2)+\alpha\epsilon}}{n}\\
& = \ln\frac{n^{\alpha/\min(\alpha,\:2)+\alpha\epsilon}}{n}\\
& = \ln\frac{n^{1+\delta+\alpha\epsilon}}{n}, \quad \text{ since } \frac{\alpha}{\min(\alpha,\:2)}=1+\delta, \text{ where } \delta \geq 0\\
& \geq \ln n^{\delta+\alpha\epsilon}\\
& = \left(\delta + \alpha\epsilon\right)\ln n \rightarrow \infty, 
\end{align*}
when $n \rightarrow \infty$.
\end{proof}
Now, after fixing $\mu$, we are going to show that 
\begin{align}
\mathcal{K}(\mu,x) & \leq \exp\Big(n\mathcal{T}(\alpha,\,\mu) + \left(Vne^{\mu x}x^{-\alpha}-\mu x \right)\Big) \leq \exp\left(o(1)+V-\alpha\epsilon\ln n\right), 
\end{align}
when $n \rightarrow \infty$, by analyzing each term separately. And so we have

\begin{lemma}
\label{lemma:P1}
Let 
$$
\mathcal{T}(\alpha,\,\mu) = \begin{cases} 
O\left(\mu^{\alpha}\right), &\text{ when } 1<\alpha<2,\\
O\left(\mu^2\ln \mu\right), &\text{ when } \alpha=2,\\
O\left(\mu^{2}\right), &\text{ when } \alpha>2,
\end{cases}
$$ 
where $\mu=\frac{1}{x}\ln\frac{x^{\alpha}}{n}$, $x = n^{\max(1/\alpha,\:1/2)+\epsilon}$ and any $\epsilon>0$. Then
$$
n\mathcal{T}(\alpha,\,\mu) = o(1),
$$
when $n \rightarrow \infty$.
\end{lemma}
\begin{proof}
First, let us consider the case when $1 < \alpha < 2$. Then $x = n^{1/\alpha+\epsilon}$ and $\mathcal{T}(\alpha,\,\mu) = O\left(\mu^{\alpha}\right)$. Therefore, we have
\begin{align*}
n\mathcal{T}(\alpha,\,\mu) & = nO\left(\mu^{\alpha}\right)\\
& = O\left( \frac{n}{x^{\alpha}} \ln^{\alpha} \frac{x^{\alpha}}{n} \right)\\
& = O\left( \frac{n}{n^{1+\alpha\epsilon}} \ln^{\alpha} \frac{n^{1+\alpha\epsilon}}{n} \right)\\
& = O\left( n^{-\alpha\epsilon} \ln^{\alpha} n^{\alpha\epsilon} \right)\\
& = o(1).
\end{align*}
Next, when $\alpha=2$. Then $x = n^{1/2+\epsilon}$ and $\mathcal{T}(\alpha,\,\mu) = O\left(-\mu^2\ln \mu\right)$. Hence,
\begin{align*}
n\mathcal{T}(\alpha,\,\mu) & = O\left(-n\mu^2\ln \mu\right)\\
& = O\left(-\frac{n}{x^2}\ln^{2} \frac{x^2}{n}\cdot\ln \left(\frac{1}{x}\ln\frac{x^2}{n}\right) \right)\\
& = O\left(\frac{n}{n^{1+2\epsilon}}\ln^{2}\frac{n^{1+2\epsilon}}{n}\ln n^{\frac{1}{2}+\epsilon} - \frac{n}{n^{1+2\epsilon}}\ln^{2} \frac{n^{1+2\epsilon}}{n}\cdot\ln\ln\frac{n^{1+2\epsilon}}{n} \right)\\
& = O\left(n^{-2\epsilon}\ln^{2}n^{2\epsilon}\ln n -n^{-2\epsilon}\ln^{2} n^{2\epsilon}\cdot\ln \ln n^{2\epsilon}\ln n \right)\\
& = O\left(n^{-2\epsilon}\ln^{2} n^{2\epsilon}\ln n \right)\\
& = o(1).
\end{align*}
Finally, when $\alpha > 2$, then $x = n^{1/2+\epsilon}$, $\mathcal{T}(\alpha,\,\mu) = O\left(\mu^{2}\right)$, while
\begin{align*}
n\mathcal{T}(\alpha,\,\mu) & = O\left(n\mu^{2}\right)\\
& = O\left( \frac{n}{x^2}\ln^2\frac{x^2}{n} \right)\\
& = O\left( \frac{n}{n^{1+2\epsilon}}\ln^2\frac{n^{1+2\epsilon}}{n} \right)\\
& = O\left( n^{-2\epsilon}\ln^2 n^{2\epsilon} \right)\\
& = o(1).
\end{align*} 
And this completes the proof of the lemma.
\end{proof}

However, unlike the $n\mathcal{T}(\alpha,\,\mu)$ term in $\mathcal{K}(\mu,x)$~\eqref{eq:K-tmp}, the term $Vne^{\mu x}x^{-\alpha}-\mu x \rightarrow -\infty$, when $n \rightarrow \infty$, as the next lemma verifies this.
\begin{lemma}
\label{lemma:P4}
Let $\alpha  > 1$ and $x = n^{\max(1/\alpha,\: 1/2)+\epsilon}$. Then
$$
Vne^{\mu x}x^{-\alpha}-\mu x  \leq V - \alpha\epsilon\ln n.
$$
\end{lemma}

\begin{proof}
We have
\begin{align*}
Vne^{\mu x}x^{-\alpha}-\mu x & = Vne^{\frac{x}{x}\ln \frac{x^{\alpha}}{n}}x^{-\alpha}-\frac{x}{x}\ln \frac{x^{\alpha}}{n} \\
& = V - \ln \frac{x^{\alpha}}{n}.
\end{align*}

When $1<\alpha\leq 2$, then $x=n^{1/\alpha+\epsilon}$, and so
\begin{align*}
V - \ln \frac{x^{\alpha}}{n} & = V - \ln \frac{n^{1+\alpha\epsilon}}{n}\\
& = V - \ln n^{\alpha\epsilon}\\
& = V-\alpha\epsilon\ln n.
\end{align*}

When $\alpha>2$, then $x=n^{1/2+\epsilon}$, and
\begin{align*}
V - \ln \frac{x^{\alpha}}{n} & = V - \ln \frac{n^{\alpha/2+\alpha\epsilon}}{n}\\
& \leq V - \ln \frac{n^{1+\alpha\epsilon}}{n}\\
& = V - \ln n^{\alpha\epsilon}\\
& = V-\alpha\epsilon\ln n.
\end{align*}

Thus, after combining both cases, we obtain that $Vne^{\mu x}x^{-\alpha}-\mu x   \leq V - \alpha\epsilon\ln n$, and the lemma follows.

\end{proof}

Next, after collecting results of Lemmas~\ref{lemma:P1}, and~\ref{lemma:P4}, we obtain from~\eqref{eq:K-tmp} that
\begin{align*}
\mathcal{K}(\mu,x) \leq \exp\Big(n\mathcal{T}(\alpha,\,\mu) + \left(Vne^{\mu x}x^{-\alpha}-\mu x \right)\Big) \leq \exp\Big(o(1)+ V-\alpha\epsilon\ln n\Big) \leq e^{2V}n^{-\alpha\epsilon},
\end{align*}
and, therefore,~\eqref{eq:K} simplifies to
\begin{align*}
\Pr[S_{n} \geq \E S_n + x] & \leq nVx^{-\alpha} + \mathcal{K}(\mu,x)\\
& \leq nVx^{-\alpha} + \exp\Big(2V - \alpha\epsilon\ln n\Big) \\
& = nVx^{-\alpha} + e^{2V}n^{-\alpha\epsilon}.
\end{align*} 
And finally, since $x=n^{\max(1/\alpha, \, 1/2)+ \epsilon}$, we obtain that
\begin{align*}
\Pr[S_{n} \geq \E S_n + x] & = \Pr[S_{n} - \E S_n \geq n^{\max(1/\alpha, \, 1/2)+ \epsilon}] \\
& \leq nVx^{-\alpha} + e^{2V}n^{-\alpha\epsilon}\\
& = Vn^{1-\max(1, \, \alpha/2)-\alpha\epsilon} + e^{2V}n^{-\alpha\epsilon},
\end{align*} 
 which proves Theorem~\ref{thr:bound_3}.
\end{proof}

Theorem~\ref{thr:bound_3} implies a simple corollary:
\begin{corollary}
\label{cor:right-tail}
Let $S_n=\sum_{i=1}^{n}X_i$, where $X_i \sim \mathbb{D}\left(\alpha,\alpha \right)$ are independent not necessary identically distributed integer-valued random variables, with $\alpha > 1$. Then w.h.p. 
$$
S_n - \E S_n \leq Cn^{\max(1/\alpha,1/2)},
$$
where  $C>0$ is some constant.
\end{corollary}

To find the left tail bounds of the r.v. $S_n-\E S_n$, we will apply the method we used while proving Theorem~\ref{thr:bound_2}, i.e. we  introduce ''inverted`` random variables to which we  apply the right-tail bound from Theorem~\ref{thr:bound_3}:

\begin{theorem}
	\label{thr:bound_4}
	Let $S_n=\sum_{i=1}^{n}X_i$, where $X_i \sim \mathbb{D}\left(\alpha_{l,i},\alpha_{r,i} \right)$ are independent not necessary identically distributed integer-valued random variables with $\alpha_{l,i},\alpha_{r,i} > 1$. Then for any $\epsilon > 0$, we have
	$$
	\Pr\left[S_n-\E S_n \leq -n^{\max(1/\alpha,\,1/2) + \epsilon}\right]  \leq Wn^{1-\max(1,\,\alpha/2) - \alpha\epsilon} + e^{2W}n^{-\alpha\epsilon},
	$$
	when $n \rightarrow \infty$.
\end{theorem}
\begin{proof}
	 Let's introduce random variables $X_i^{'}$ that have the same distributions as $-X_i$, i.e. $X_i^{'} \overset{d}{=} -X_i$. Clearly, $X_i^{'} \sim \mathbb{D}\left( \alpha_{r,i}, \alpha_{l,i}\right)$ with $\alpha_{l,i}, \alpha_{r,i} > 1$, and so, by Definition~\ref{def:D}, $F_{X_i'+}(k)\leq W_{X_i}k^{-\alpha_{l,i}}$. 
	 
	 Since $\alpha_{l,i}, \alpha_{r,i} > 1$, from Lemma~\ref{lemma-expectation} it follows that $\left|\E X_i^{'}\right| < \infty$. Then
	\begin{align*}
	\Pr\left[S_n -\E S_n\leq -x\right] & = \Pr\left[ \sum_{i=1}^{n}\left(X_i-\E X_i\right) \leq -x\right] \\
	& = \Pr\left[ \sum_{i=1}^{n}\left(\E X_i - X_i\right) \geq x\right]\\
	& = \Pr\left[ \sum_{i=1}^{n}\Big(-X_i - \E\left[-X_i\right]\Big) \geq x\right]\\
	& = \Pr\left[ \sum_{i=1}^{n}\Big(X_i^{'}-\E X_i^{'}\Big) \geq x\right]\\
	& = \Pr\left[ \sum_{i=1}^{n}X_i^{'}-\sum_{i=1}^{n}\E X_i^{'} \geq x\right]\\
	& = \Pr\left[ S_n^{'} -\E S_n^{'}\geq x\right], \text{ where } S_n^{'}:=\sum_{i=1}^{n}X_i^{'}\\
	& \leq nWx^{-\alpha} + e^{2W}n^{-\alpha\epsilon}.
	\end{align*}
	 The last inequality here follows after applying Theorem~\ref{thr:bound_3} to the sum $S_n^{'}-\E S_n^{'}$,  which consists of random variables each of which has  tail  distribution functions that can be bounded by some $Ck^{-\alpha}$ with constants $C>0$ and $\alpha>1$.
	 
	 Finally, since $x=n^{\max(1/\alpha, \, 1/2)+\epsilon}$, we obtain that
\begin{align*}
\Pr\left[S_n -\E S_n\leq -x\right] & = \Pr\left[S_n -\E S_n\leq -n^{\max(1/\alpha, \, 1/2)+\epsilon}\right] \leq Wn^{1-\max(1, \, \alpha/2)-\alpha\epsilon} +  + e^{2W}n^{-\alpha\epsilon},
\end{align*}
	 and the theorem follows.
	 
\end{proof}

As a result from Theorem~\ref{thr:bound_4} and Corollary~\ref{cor:right-tail} another useful corollary follows:
\begin{corollary}
	\label{cor:concentration}
	Let $S_n=\sum_{i=1}^{n}X_i$, where $X_i \sim \mathbb{D}\left(\alpha,\,\alpha \right)$ are independent not necessary identically distributed integer-valued random variables, with $\alpha > 1$. Then w.h.p. 
	$$
	|S_n - \E S_n| \leq Cn^{\max(1/\alpha,\:1/2)},
	$$
	where $C>0$ is some constant.
\end{corollary}

Hence, when $S_n=\sum_{i=1}^{n}X_i$ consists of variables, whose tail functions can be bounded by function $C \, k^{-\alpha}$ with $\alpha>1$ and some constant $C>0$, then we do not expect $S_n$ to deviate much from its expected value $\E S_n$. 

\bibliographystyle{plain}
\bibliography{sums} 

\end{document}